\documentclass[a4paper]{amsart}
\usepackage{amssymb}
\usepackage[all]{xy}
\usepackage{hyperref, aliascnt}
\usepackage{bbm}
\usepackage{stmaryrd}

\newcounter{theoremintro}
\newtheorem{thmintro}[theoremintro]{Theorem}

\newtheorem{lma}{Lemma}[section]

\newaliascnt{thmCt}{lma}
\newtheorem{thm}[thmCt]{Theorem}
\aliascntresetthe{thmCt}

\newaliascnt{corCt}{lma}
\newtheorem{cor}[corCt]{Corollary}
\aliascntresetthe{corCt}

\newaliascnt{prpCt}{lma}
\newtheorem{prp}[prpCt]{Proposition}
\aliascntresetthe{prpCt}

\newtheorem*{thm*}{Theorem}
\newtheorem*{cor*}{Corollary}
\newtheorem*{prop*}{Proposition}

\theoremstyle{definition}

\newaliascnt{pgrCt}{lma}
\newtheorem{pgr}[pgrCt]{}
\aliascntresetthe{pgrCt}

\newaliascnt{dfnCt}{lma}
\newtheorem{dfn}[dfnCt]{Definition}
\aliascntresetthe{dfnCt}

\newaliascnt{rmkCt}{lma}
\newtheorem{rmk}[rmkCt]{Remark}
\aliascntresetthe{rmkCt}

\newaliascnt{rmksCt}{lma}

\aliascntresetthe{rmksCt}

\newaliascnt{egCt}{lma}
\newtheorem{eg}[egCt]{Example}
\aliascntresetthe{egCt}

\newaliascnt{qstCt}{lma}

\aliascntresetthe{qstCt}

\newaliascnt{pbmCt}{lma}

\aliascntresetthe{pbmCt}

\newaliascnt{ntnCt}{lma}
\newtheorem{ntn}[ntnCt]{Notation}
\aliascntresetthe{ntnCt}

\def\today{\number\day\space\ifcase\month\or   January\or February\or
   March\or April\or May\or June\or   July\or August\or September\or
   October\or November\or December\fi\   \number\year}

\DeclareMathOperator{\Lt}{L}
\DeclareMathOperator{\Rt}{R}
\DeclareMathOperator{\Gl}{Gl}

\newcommand{\NN}{{\mathbb{N}}}

\newcommand{\CC}{{\mathbb{C}}}
\newcommand{\RR}{{\mathbb{R}}}
\newcommand{\TT}{{\mathbb{T}}}

\newcommand{\Isom}{{\mathrm{Isom}}}
\newcommand{\id}{{\mathrm{id}}}

\newcommand{\im}{{\mathrm{Im}}}
\newcommand{\re}{{\mathrm{Re}}}
\newcommand{\supp}{{\mathrm{supp}}}
\newcommand{\Aut}{{\mathrm{Aut}}}
\newcommand{\tensMax}{\hat{\otimes}}
\newcommand{\ca}{$C^*$-algebra}
\newcommand{\lcg}{locally compact group}
\newcommand{\Bdd}{\mathcal{B}}

\newcommand{\frakA}{\mathfrak{A}}
\newcommand{\weakStar}{weak${}^*$}
\newcommand{\Unitary}{\mathcal{U}}
\newcommand{\charFct}{\mathbbm{1}}
\newcommand{\andSep}{\quad \text{and}\quad}
\newcommand{\PF}{PF}
\newcommand{\PM}{PM}
\newcommand{\CV}{CV}
\newcommand{\level}[1]{\llbracket #1 \rrbracket}
\newcommand{\llevel}[1]{\left\llbracket #1 \right\rrbracket}

\begin{document}

\title{Isomorphisms of Algebras of Convolution Operators}

\author[Eusebio Gardella]{Eusebio Gardella}
\address{Eusebio Gardella
Mathematisches Institut, Fachbereich Mathematik und Informatik der
Universit\"at M\"unster, Einsteinstrasse 62, 48149 M\"unster, Germany.}
\email{gardella@uni-muenster.de}
\urladdr{http://www.math.uni-muenster.de/u/gardella/}

\author{Hannes Thiel}
\address{Hannes Thiel
Mathematisches Institut, Fachbereich Mathematik und Informatik der
Universit\"at M\"unster, Einsteinstrasse 62, 48149 M\"unster, Germany.}
\email{hannes.thiel@uni-muenster.de}
\urladdr{http://www.math.uni-muenster.de/u/hannes.thiel/}


\thanks{The first named author was partially supported by the D.~K. Harrison Prize from the University of Oregon,
and by a postdoctoral fellowship from the Humboldt Foundation.
Both authors were partially supported by the Deutsche Forschungsgemeinschaft (SFB 878 Groups, Geometry \& Actions).}



\begin{abstract}
For $p,q\in [1,\infty)$, we study the isomorphism problem for
the $p$- and $q$-convolution algebras associated to locally compact groups.
While it is well known that not every group can be recovered from its
group von Neumann algebra, we show that this is the case for the algebras
$\CV_p(G)$ of $p$-convolvers and $\PM_p(G)$ of $p$-pseudomeasures, for $p\neq 2$.
More generally, we show that if $\CV_p(G)$ is isometrically
isomorphic to $\CV_q(H)$, with $p,q\neq 2$, then $G$ must be isomorphic to
$H$ and $p$ and $q$ are either equal or conjugate. This implies that there 
is no $L^p$-version of Connes' uniqueness of the hyperfinite II$_1$-factor. 
Similar results apply to the algebra $\PF_p(G)$ of $p$-pseudofunctions, generalizing a classical result of Wendel.
We also show that other $L^p$-rigidity results for groups can be easily
recovered and extended using our main theorem.

Our results answer questions originally formulated in the work of Herz in the 70's.
Moreover, our methods reveal new information about the Banach algebras in
question. As a non-trivial application, we verify the reflexivity conjecture for
all Banach algebras lying between $\PF_p(G)$ and $\CV_p(G)$: if any such algebra
is reflexive and amenable, then $G$ is finite. 
\end{abstract}

\maketitle


\renewcommand*{\thetheoremintro}{\Alph{theoremintro}}
\section{Introduction}
\label{sec:Intro}
Convolution algebras of groups are among the most important and widely studied examples of Banach algebras.
For a locally compact group $G$, the algebras that have arguably received the greatest attention are
the function algebra $L^1(G)$, the measure algebra $M(G)$, the reduced group \ca\ $C^*_\lambda(G)$, and
the group von Neumann algebra $L(G)$. A significant part of the literature in this area has focused on
identifying those properties of a group that are reflected on its convolution algebras. An early and
illustrative instance of this is Johnson's celebrated result \cite{Joh72CohomologyBAlg} asserting that a locally compact
group $G$ is amenable if and only if $L^1(G)$ is amenable as a Banach algebra. Similar results have
been obtained for $M(G)$, $C_\lambda^*(G)$ and $L(G)$, at least when $G$ is discrete.

The isomorphism problem in Harmonic Analysis asks to determine when two groups have isometrically isomorphic
convolution algebras. This problem has received a great deal of attention, particularly in what refers to
identifying those groups that can be \emph{recovered} from one of their convolution algebras. The first
result in this direction is Wendel's classical result \cite{Wen51IsometrIsoGpAlgs}, asserting that $L^1(G)$ is isometrically
isomorphic to $L^1(H)$ if and only if $G$ is (topologically) isomorphic to $H$. Shortly after, Johnson
used Wendel's theorem to prove a similar result for the measure algebra \cite{Joh_isometric_1964}. The situation for the
operator algebras $L(G)$ and $C^*_\lambda(G)$ is, however, more complicated. For once, not every group can be recovered from
its von Neumann algebra, or even from its reduced group \ca: consider, for example, $\mathbb{Z}_2\oplus\mathbb{Z}_2$ and
$\mathbb{Z}_4$. More drastically, Connes' celebrated result on uniqueness of the hyperfinite II$_1$-factor
implies that any two countable amenable ICC groups have isomorphic group von Neumann algebras. 
There also exist groups that can be recovered from their $C^*$-algebras but not from
their von Neumann algebras (such as $\mathbb{Z}$). Positive results in this context (usually referred to as
\emph{superrigidity} results) are difficult to find: for von Neumann algebras, the first one is
the groundbreaking work of Ioana-Popa-Vaes (\cite{IoaPopVae_class_2013}) on wreath product groups, while non-trivial results for
$C^*$-algebras are even more recent (\cite{EckRau_superrigidity_2018, KnuRauThiWhi8pre:RigidVirtAbln}). Indeed, the passage
from $C^*_\lambda(G)$ to $L(G)$ tends to ``erase'' a lot of information about $G$. For example, while every countable,
torsion-free, abelian group is recovered from its reduced group $C^*$-algebra, all such groups have isomorphic
group von Neumann algebras. Similarly, while it is known that $\mathbb{F}_n$ and $\mathbb{F}_m$ have
non-isomorphic group $C^*$-algebras, whether this is the case for their von Neumann algebras is a notable
open problem.

In this work, we are interested in the $L^p$-version of the problems described above. Given $p\in [1,\infty)$
and a locally compact group $G$, we let $\lambda_p$ be the representation of $L^1(G)$ on $L^p(G)$
given by left convolution. The algebra of \emph{$p$-pseudofunctions} $\PF_p(G)$ is the Banach algebra generated
by $\lambda_p(L^1(G))$; the algebra of \emph{$p$-convolvers} $\CV_p(G)$ is the double commutant of $\PF_p(G)$;
  and, for $p>1$, the algebra of \emph{$p$-pseudomeasures} $\PM_p(G)$ is the \weakStar-closure\footnote{Identifying
$\mathcal{B}(L^p(G))$ with the dual of $L^p(G)\widehat{\otimes}L^{p'}(G)$ in a canonical way.} of $\PF_p(G)$.
We thus obtain a continuously varying family of Banach algebras which for $p=1$ give the
algebras $L^1(G)$ and $M(G)$, and for $p=2$ agree with $C^*_\lambda(G)$ and $L(G)$.

These algebras were introduced by Herz \cite{Her73SynthSubgps} several decades ago, and have been
intensively studied by a number of authors; see \cite{Her76AsymNormsConv, CowFou76InclNonincl, Der11ConvOps, DawSpr14arX:ApproxPropConvPM}.
In recent years, the influx of operator-algebraic techniques has given the area new impetus; see
\cite{Phi13arX:LpCrProd, Phi14pre:LpMultDom, GarThi15GpAlgLp, GarThi18:ReprConvLq}. Among these, we
mention the solution \cite{GarThi16QuotBAlgLp} to a long-standing open problem of Le Merdy: for $p\neq 2$, there exists a quotient
of $\PF_p(\mathbb{Z})$ which cannot be represented on an $L^p$-space.

Despite the advances, a number of questions remain open, not
least due to the difficulties in understanding the geometry of $L^p$-spaces (by comparison with the case $p=2$).
A significant problem in the area, known as the \emph{convolvers and pseudomeasures problem} and originally raised
by Herz, asks to determine if the double-commutant theorem holds for $\PM_p(G)$, that is, whether it is always true that $\PM_p(G)=\CV_p(G)$.
This is known to be the case for $p=2$ (regardless of $G$), and whenever $G$ has the approximation
property (regardless of $p$); see \cite{Cow_predual_1998, DawSpr14arX:ApproxPropConvPM}.

Another problem, also raised by Herz, is the isomorphism question
for these algebras. In its most general
form, the problem is to determine, for $p,q\in [1,\infty)$ and locally compact groups $G$ and $H$, when there exists
an isometric isomorphism $\CV_p(G)\cong \CV_q(H)$ (or $\PM_p(G)\cong \PM_q(H)$, or $\PF_p(G)\cong \PF_q(H)$).
Earlier results \cite{Her73SynthSubgps,Der11ConvOps} mostly focused on duality theory and offered partial answers
when $G=H$; this case was recently
settled by the authors in \cite{GarThi18:ReprConvLq}. In this work, we focus on the remaining and arguably most
difficult part of this problem, namely deciding when two groups have isometrically isomorphic $p$-convolution
algebras.

The outcome is not a priori clear. Indeed, $p$-convolution algebras tend to behave more
like the case $p=2$ (group $C^*$-algebras and von Neumann algebras) than like the case $p=1$ (function and
measure algebras), mostly thanks to either reflexivity, uniform convexity, or interpolation for operators on
$L^p$-spaces.
Some examples of this are as follows:
\begin{itemize}
 \item For a Powers group $G$ and $p\in (1,\infty)$, the Banach algebra $\PF_p(G)$ is simple,
 while this fails for $p=1$; see \cite{PoyHej14arX:SimpleLp}.
 \item For $p\in (1,\infty)$, amenability of $G$ is characterized by the fact that $\PF_p(G)$ is universal with respect to
 representations of $G$ on $L^p$-spaces, while this fails for $p=1$; see \cite{GarThi15GpAlgLp}.
 \item There exists, for $p\in (1,\infty)$, an analog of Choi's multiplicative domain theorem from \cite{Cho_schwarz_1974}
 for $p$-completely contractive maps from $\PF_p(G)$, $\PM_p(G)$ or $\CV_p(G)$, while this fails for $p=1$;
 see \cite{Phi14pre:LpMultDom}.
 \item For a discrete group $G$ with trivial amenable radical and $p\in (1,\infty)$, the Banach algebras
 $\PF_p(G)$, $\PM_p(G)$ and $\CV_p(G)$ have a unique (\weakStar-continuous) tracial state\footnote{Since this
 result is not available in the literature, we outline its proof here.
 By Theorem~5.2 in \cite{Haa_new_2016}, triviality of the amenable radical is equivalent to a norm inequality in $\Bdd(\ell^2(G))$
 for a convex combination of unitaries coming from the left regular representation. Using the Riesz-Thorin interpolation theorem, it is
 easy to see that this condition is equivalent to an identical norm inequality in $\Bdd(\ell^p(G))$, for $p\in (1,\infty)$.
 Once this is established, it follows that any linear functional on $\PF_p(G)$ which is ``unitary invariant" must
 vanish on the non-trivial group elements. From this, uniqueness of the trace on $\PF_p(G)$ can be deduced immediately.
 The arguments for \weakStar-continuous traces on $\CV_p(G)$ and $\PM_p(G)$ are analogous.},
 while this fails for $p=1$.
\end{itemize}

It is therefore not clear whether one should expect a complete rigidity phenomenon as in the results of
Wendel-Johnson for $p=1$, a very restricted form of rigidity as in the case of operator algebras, or something
intermediate. As it turns out, every locally compact group is $L^p$-rigid, for $p\neq 2$, in a strong sense.
In fact, more can be said, and a particular case of our main result can be stated as follows.

\begin{thmintro}\label{thmintro:rigidity} (\autoref{prp:MorCVp:IsoGpHq})
Let $G$ and $H$ be locally compact groups, and let $p,q\in [1,\infty)$ be not both equal to 2.
Suppose there is an isometric isomorphism
$\CV_p(G)\cong \CV_q(H)$ (or $\PM_p(G)\cong \PM_q(H)$, or $\PF_p(G)\cong \PF_q(H)$).
Then $G$ is isomorphic to $H$ and $p$ and $q$ are either equal or conjugate.
\end{thmintro}

We want to highlight the striking fact that even the very large \weakStar-closed algebras $\CV_p(G)$ and $\PM_p(G)$
remember $G$ completely, which shows a stark contrast with the case $p=2$ of group von Neumann algebras discussed above.
In particular, the theorem above rules out any $L^p$-analog of Connes' uniqueness of the hyperfinite II$_1$-factor, 
for $p\neq 2$.

There exist other results on $L^p$-rigidity of groups, although of a rather different flavor.
Indeed, Wendel's appraised result motivated a number of authors to search for $L^p$ analogs of it. A series of
intermediate results culminated in Parrott's theorem (\cite[Theorem~2]{Par68IsoMultiplier}) asserting that for locally
compact groups $G$ and $H$, and for $p\in (1,\infty)\setminus\{2\}$, if there is an invertible linear
isometry $T\colon L^p(G)\to L^p(H)$
such that $T(\xi\ast\eta)=T(\xi)\ast T(\eta)$ whenever $\xi,\eta,\xi\ast\eta\in L^p(G)$, and
similarly for $T^{-1}$, then $G$ and $H$
are isomorphic. In \autoref{prp:Parrott}, we show that our Theorem~\ref{thmintro:rigidity} is formally
stronger, by deducing Parrott's result as a corollary.

We briefly discuss the proof of Theorem~\ref{thmintro:rigidity}.
In Sections~3 and~4 we set the stage and
prove a version of the Banach-Lamperti theorem from \cite{Lam58IsoLp} which is applicable to the Haar measure
of an arbitrary locally compact group which is not necessarily $\sigma$-finite; see
\autoref{prp:Lamperti:MainResult}.
(An additional reason to do this is the fact,
pointed out in \cite{CamFauGar79IsomLpCopiesLpShift} and \cite{GioPes07ExtrAmenGps}, that Lamperti's
characterization of the invertible isometries, and its proof, are not entirely correct.) This description is then applied
to the invertible isometries in $\CV_p(G)$, for $p\neq 2$, obtaining, for each of them, a
Boolean automorphism of the $\sigma$-algebra of $G$
which commutes with right translations. The next step requires overcoming the first major difficulty
in this work,
namely proving that such a transformation can be lifted to a measurable map $G\to G$, and that this
map is continuous. This relies on a new lattice-theoretic characterization (see \autoref{prp:UCVp:charOpen})
of those measurable subsets of $G$ which are open modulo null-sets. The next step is proving that
the resulting continuous map $G\to G$ is given by a scalar multiple of the left translation by a group element (see
\autoref{prp:UCVp:identifyUCVp}). It should be pointed out that these results become relatively
easy to prove when $G$ is discrete, and we provide an argument in this case for the convenience of
the reader (see \autoref{lma:DiscreteCase}). The result is thus an \emph{algebraic} identification
between the invertible isometries of $\CV_p(G)$ and $\mathbb{T}\times G$.
The next significant obstacle in our work is recovering the topology of $G$ intrinsically from the
norm topology of $\CV_p(G)$. The first result in this direction is as follows (see \autoref{ntn:UCVp:pi0}).

\begin{thmintro}\label{thmintro:IsoWkTop} (\autoref{prp:UCVp:TopUCVp})
Let $G$ be a locally compact group, and let $p\in (1,\infty)$ with $p\neq 2$.
Then there are natural isomorphisms of topological groups
\[
\TT\times G\cong \Unitary(\CV_p(G))_{w^*},\andSep  G\cong \pi_0(\Unitary(\CV_p(G)))_{w^*}.
\]
\end{thmintro}

In other words, the topology on $G$ can be recovered using the \weakStar-topology on
$\CV_p(G)$. Since it is not known whether $\CV_p(G)$ has a unique predual, this by itself is not
enough to recover the topology of $G$ from the \emph{norm} topology on $\CV_p(G)$. 
However, it turns out that any isometric isomorphism of $\CV_p(G)$ is \emph{automatically}
\weakStar-continuous when restricted to the group $\Unitary(\CV_p(G))$ of invertible isometries.
This implies Theorem~\ref{thmintro:rigidity}
for the algebra of $p$-convolvers, while a similar argument applies to the algebra of $p$-pseudomeasures.

On the other hand, if we start from an isometric isomorphism between the algebras of $p$-pseudofunctions,
we show that the topology of $G$ can be recovered using
the strict topology in the left multiplier algebra $M_l(\PF_p(G))$ of
$\PF_p(G)$. More explicitly, we show
that there exists a canonical isomorphism of topological groups $G\cong \pi_0(\mathcal{U}(M_l(\PF_p(G))))_{\mathrm{str}}$;
see \autoref{prp:UMFp:MainResult}.

It should be pointed out that our proof is quite different from Wendel's: while he identified the left centralizers
of $L^1(G)$ with the group $G$, we focus on the invertible isometries of $\CV_p(G)$.
Moreover, Wendel's argument relies heavily on the adjoint operation on $L^1(G)$ through the notion of positivity
for homomorphisms between group algebras and Kawada's description 
of such maps.
Finally, the greatest difficulties in our work, which come from the non-triviality of the
topology on $G$, do not appear at all when $p=1$, since
Kawada's proof reduces to the discrete case by passing to the measure algebra $M(G)$, where
the elements of $L^1(G)$ have the $\ell^1$-norm.

We would also like to point out that the intermediate results obtained in this work are bound to
have applications in problems not related to isomorphisms of convolution algebras, particularly
\autoref{prp:UCVp:charOpen} and Theorem~\ref{thmintro:IsoWkTop}. In the last section, we give
a nontrivial
application of the latter. There, we verify the reflexivity conjecture of
Gal\'e-Ransford-White \cite{GalRanWhi_weakly_1992} for $p$-convolution algebras:

\begin{thmintro}\label{thmintro:Reflexive} (\autoref{prp:ReflMain})
Let $G$ be a locally compact group, and let $p\in (1,\infty)$. If any Banach algebra
containing $\PF_p(G)$ and contained in $\CV_p(G)$ is reflexive and amenable, then $G$ is finite.
\end{thmintro}

This result strengthens and extends previously known ones,
where amenability of $G$ needed to be assumed instead.
Our proof uses a number of results and ideas from the literature. First, reflexivity is used
in conjunction with the work of Daws~\cite{Daw07DualBAlgReprInj} 
and~\autoref{prp:IsoCVp:IsoWkClosed} to show that $G$ must be discrete. Amenability of $G$
is deduced using ideas of Phillips, while results of Runde~\cite{Run_representations_2005} 
and the authors \cite{GarThi15GpAlgLp} are used to show that $G$ must be compact, and 
hence finite. 

\subsection*{Acknowledgements}
The authors would like to thank Siegfried Echterhoff, Chris Phillips, and Nico Spronk for helpful conversations.
We also thank Sven Raum for helpful feedback on a preliminary version of this work. 

\section{The Radon-Nikodym theorem for valuations on measure algebras}

In this short section, we set the stage for proving a version of the Banach-Lamperti Theorem
for localizable measure algebras (\autoref{prp:Lamperti:MainResult}) by deriving a
Radon-Nikodym Theorem in this setting (\autoref{prp:Lamperti:RNderivative_levels} and
\autoref{prp:Lamperti:RNderivative_integral}). These results are probably known to the experts, but we
have not been able to find them in the literature. For the convenience of the reader,
we give only abbreviated proofs.

We first recall basic notions from the theory of measure algebras.
For details we refer to the books of Fremlin, \cite{Fre-MsrThy3a, Fre-MsrThy3b}.

\begin{pgr}
\label{pgr:valuation}
Let $\frakA$ be a Boolean algebra.
A \emph{valuation} on $\frakA$ is a map $\mu\colon\frakA\to[0,\infty]$ that is strict 
($\mu(0)=0$) and additive (that is, $\mu(a\vee b)=\mu(a)+\mu(b)$ whenever $a,b\in\frakA$ are orthogonal).
Note that these two conditions are equivalent to the usual definition of a valuation as 
a map that is strict, order-preserving, and modular (that is, $\mu(a\vee b)+\mu(a\wedge b)=\mu(a)+\mu(b)$ for all $a,b\in\frakA$).


We call $\mu$ \emph{strictly positive} if $\mu(a)>0$ for every $a\in\frakA$ with $a\neq 0$, and 
we call $\mu$ \emph{semi-finite} if for every $a\in\frakA$ with $\mu(a)=\infty$ there exists $a'\in\frakA$ with $a'\leq a$ and $0<\mu(a')<\infty$.



Recall that $\frakA$ is said to be \emph{complete} (\emph{$\sigma$-complete}) if every nonempty (countable) subset of $P$ with an upper bound has a supremum and every nonempty (countable) subset of $P$ with a lower bound has an infimum.

If $\frakA$ is complete ($\sigma$-complete), then we call $\mu$ \emph{completely additive} (\emph{$\sigma$-additive}) if $\mu(\sum_{j\in J}b_j)=\sum_{j\in J}\mu(b_j)$ for every (countable) family $(b_j)_{j\in J}$ of pairwise orthogonal elements in $\frakA$.

\end{pgr}

The following is the main notion of this section;
see \cite[Definition~321A, p.68]{Fre-MsrThy3a}.

\begin{dfn}
A \emph{measure algebra} $(\frakA,\mu)$ is a $\sigma$-complete Boolean algebra $\frakA$ together with a strictly positive, $\sigma$-additive valuation $\mu\colon\frakA\to[0,\infty]$.
A measure algebra $(\frakA,\mu)$ is said to be \emph{localizable} if $\frakA$ is complete and $\mu$ is semi-finite.
\end{dfn}

The motivating example comes from measure spaces, in the following sense.

\begin{eg}
Let $\mu=(X,\Sigma,\mu)$ be a measure space. If $\mathcal{N}$ denotes the
family of null sets, then the quotient $\Sigma/\mathcal{N}$ is a $\sigma$-complete Boolean algebra.
Further, $\mu$ induces a map $\bar{\mu}\colon\Sigma/\mathcal{N}\to[0,\infty]$ that maps the class of $E\in\Sigma$ to $\mu(E)$.
Then $(\Sigma/\mathcal{N},\bar{\mu})$ is a measure algebra, called the \emph{measure algebra} associated to $\mu$.
By \cite[Theorem~322B, p.72]{Fre-MsrThy3a}, a measure space is localizable (in the sense of \cite[Definition~211G, p.13]{Fre-MsrThy2}) if and only if its assoicated measure algebra is.
\end{eg}

\begin{pgr}
\label{pgr:L0}
Recall that the set $\mathfrak{B}(\RR)$ of Borel subsets of $\RR$ is order-complete.
Let $\frakA$ be a $\sigma$-complete Boolean algebra.
A (real-valued) \emph{measurable function} on $\frakA$ is a sequentially order-continuous Boolean homomorphism $f\colon\mathfrak{B}(\RR)\to\frakA$;
see \cite[Proposition~364F, p.99]{Fre-MsrThy3b}\footnote{If $\hat{f}$ is a measurable function from a measurable space underlying $\mathfrak{A}$,
one obtains a measurable function $f\colon \mathfrak{B}(\RR)\to \frakA$ by setting $f(E)=\hat{f}^{-1}(E)$ for all Borel subsets $E\subseteq\RR$.}.
Following Fremlin, we use $\level{f\leq t}$ to denote $f((-\infty,t])$, and analogously for $\level{f<t}$, and more generally $\level{f\in E}$ for a measurable subset $E\subseteq\RR$.
The map $f$ is determined by its values on the subsets $(-\infty,t]\subseteq\RR$ for $t\in\RR$.
It is also determined by $\level{f>t}$ for $t\in\RR$.

The space of real-valued measurable functions on $\frakA$ is denoted by $L^0_\RR(\frakA)$.
We let $L^0_\RR(\frakA)_+$ denote the subset of positive functions.
The real vector space $L^0_\RR(\frakA)$ is canonically a real, commutative algebra;
see \cite[Theorem~364D p.98]{Fre-MsrThy3b}.

We let $L^0(\frakA)$ be the vector space of complex-valued measurable functions on $\frakA$, which is formally defined as
\[
L^0(\frakA) = \big\{ f+ ig\colon f,g\in L^0_\RR(\frakA) \big\}.
\]
We denote the real and imaginary parts of $f\in L^0(\frakA)$ by $\re(f)$ and $\im(f)$, respectively.
The complex vector space $L^0(\frakA)$ is canonically a complex, commutative algebra.
Given $f\in L^0(\frakA)$, we consider $\re(f)^2+\im(f)^2$ as an element in $L^0_\RR(\frakA)$ and we define the absolute value of $f$ as the unique $|f|\in L^0_\RR(\frakA)_+$ that satisfies
\[
\level{|f|\leq t}
= \level{ \re(f)^2 + \im(f)^2 \leq t^2 },
\]
for all $t\in\RR_+$.
Similarly, for $p\in(1,\infty)$ there is a unique $|f|^p\in L^0_\RR(\frakA)_+$ satisfying
\[
\level{|f|^p\leq t}
= \level{ |f| \leq t^{1/p}},
\]
for every $t\in\RR_+$.
\end{pgr}

Next, we define the $L^p$-norms of a function on a measure algebra.

\begin{pgr}
\label{pgr:Lp}
Let $(\mathfrak{A},\mu)$ be a measure algebra, and let $f\in L^0(\frakA)$.
The function $[0,\infty)\to[0,\infty]$ given by $t\mapsto \mu(\level{f>t})$ is
decreasing and therefore Lebesgue-measurable, which allows us to define the $L^1$-norm
$\|f\|_1\in [0,\infty]$ of $f$ by
\[
\|f\|_1 = \int_0^\infty \mu(\level{|f|>t})\, dt.
\]
As usual, the space of intregrable functions on $\frakA$ is defined as
\[
L^1(\frakA,\mu)
= \big\{ f \in L^0(\frakA)\colon \|f\|_1 < \infty \big\},
\]
and it is a Banach space for the norm $\|\cdot\|_1$.
We will usually abbreviate $L^1(\frakA,\mu)$ to $L^1(\mu)$.
Given $f\in L^0_\RR(\frakA)_+$ with $\|f\|_1<\infty$, the integral of $f$ is defined as
\[
\int f d\mu = \int_0^\infty \mu(\level{f>t}) dt.
\]
This is extended to $L^1(\mu)$ linearly in the obvious way.
Given $p\in(1,\infty)$, one defines
\[
L^p(\frakA,\mu)
= \big\{ f\in L^0(\frakA)\colon  |f|^p \in L^1(\frakA,\mu) \big\},
\]
which also is a Banach space for the norm  $\|f\|_p = \| |f|^p \|_1^{1/p}$.
We usually write $L^p(\mu)$ for $L^p(\frakA,\mu)$.
For $f\in L^0(\frakA)$, we set $\|f\|_\infty = \inf\{t\geq 0\colon  \level{|f|\leq t}= 1 \}$, and
\[
L^\infty(\frakA)= \big\{ f\in L^0(\frakA) \colon  \|f\|_\infty<\infty \big\}.
\]
\end{pgr}

The next two results on Radon-Nikodym derivatives for valuations on complete Boolean algbras that will be used in \autoref{sec:Lamperti}.

\begin{lma}
\label{prp:Lamperti:RNderivative_levels}
Let $\frakA$ be a complete Boolean algebra, and let $\sigma,\mu\colon\frakA\to[0,\infty]$
be strictly positive, semi-finite, completely additive valuations.
For $t\in(0,\infty)$, the set
\[
D_t =\big\{ a\in\frakA \colon \sigma(a')\leq t\mu(a') \mbox{ whenever } a'\leq a\big\}
\]
is downward hereditary and contains a largest element $e_t$ given by $e_t=\sup D_t$.
Then there is a unique element $\tfrac{d\sigma}{d\mu}\in L_{\mathbb{R}}^0(\frakA)_+$, which we call the Radon-Nikodym derivative of $\sigma$ with respect to $\mu$, such that $\level{\tfrac{d\sigma}{d\mu}\leq t}=e_t$ for $t\in(0,\infty)$.
Further, $\tfrac{d\sigma}{d\mu}$ is strictly positive, that is, $\level{\tfrac{d\sigma}{d\mu}>0}=1$.
Finally, we have
\begin{equation}
\label{prp:Lamperti:RNderivative_levels:eq}
s\mu(a) \leq \sigma(a) \leq t\mu(a),
\end{equation}
for every $0< s \leq t<\infty$ and $a\leq\level{s\leq\frac{d\mu}{d\sigma}\leq t}$.
\end{lma}
\begin{proof}
Let $t\in(0,\infty)$.
It is straightforward to verify that $D_t$ is downward hereditary.
Set $e_t=\sup D_t$.
To verify that $e_t$ belongs to $D_t$, let $a\in\frakA$ satisfy $a\leq e_t$.
If $D_t=\{0\}$, then there is nothing to show.
Otherwise, choose a maximal family $(a_j)_{j\in J}$ of nonzero, pairwise orthogonal elements in $D_t$.
Then $e_t=\sum_{j\in J} a_j$.
(If $e_t>\sum_{j\in J} a_j$, then we can choose $x\in D_t$ with $x\nleq\sum_{j\in J}a_j$ and then $x\setminus(\sum_{j\in J}a_j)$ is nonzero, belongs to $L_t$ and is orthogonal to each $a_j$, contradicting the maximality of $(a_j)_j$.)
By distributivity in $\frakA$ (\cite[Lemma~1.33, p.22]{HdbkBooleanAlg1}), we have $a=a\wedge (\sum_{j\in J} a_j)=\sum_{j\in J} (a\wedge a_j)$.
Using at the first step that $\sigma$ is completely additive, we deduce
\[
\sigma(a)
= \sum_{j\in J} \sigma(a\wedge a_j)
\leq \sum_{j\in J} t \mu(a\wedge a_j)
\leq t \mu(\sum_{j\in J} (a\wedge a_j))
= \mu(a).
\]
Thus, $e_t$ belongs to $D_t$.

\textbf{Claim~1}:\emph{ 
Let $a\in\frakA$ and $t\in(0,\infty)$ satisfy $\sigma(a)<t\mu(a)$.
Then there exists $a'\leq e_t$ with $0\neq a'\leq a$.}
To prove the claim, choose a maximal family $(b_j)_{j\in J}$ of nonzero, pairwise orthogonal elements satisfying $b_j\leq a$ and $\sigma(b_j)\geq t\mu(b_j)$.
Set $b=\sum_{j\in J}b_j$, with the convention that $b=0$ if $J=\emptyset$.
We have $b\leq a$.
If $b=a$, then $\sigma(a)=\sum_{j\in J}\sigma(b_j)\geq\sum_{j\in J} t\mu(b_j)=t\mu(a)$, which is a contradiction.
Set $a'=a\setminus b$, which is nonzero.
To show that $a'$ belongs to $D_t$, let $a''\in\frakA$ satisfy $0\neq a''\leq a'$.
If $\sigma(a'')\nleq t\mu(a'')$, then this contradicts the maximality of the family $(b_j)_{j\in J}$.
Hence, we have $\sigma(a'')\leq t\mu(a'')$, which shows that $a'\in D_t$.

\textbf{Claim~2:}
\emph{Let $a\in\frakA$ and $t\in(0,\infty)$ satisfy $a\leq 1\setminus e_t$.
Then $\sigma(a)\geq t\mu(a)$.}
To prove the claim, assume that $\sigma(a)<t\mu(a)$.
By Claim~1, we can find $a'\leq e_t$ with $0\neq a'\leq a$, which is a contradiction.

Set $e_0=\bigwedge_{t>0}e_t$ and $e_t=0$ for $t<0$.
For every $t\in(0,\infty)$, we have $e_0\leq e_t$ and therefore $\sigma(e_0)\leq t\mu(e_0)$.
We deduce that $\sigma(e_0)=0$ or $\mu(e_0)=\infty$.
In the first case, we get $e_0=0$ since $\sigma$ is strictly positive.
In the second case, we use that $\mu$ is semi-finite to obtain $a\leq e_0$ with $0<\mu(a)<\infty$.
Since $a\leq e_t$, we again obtain $\sigma(a)\leq t\mu(a)$ for every $t\in(0,\infty)$, whence $\sigma(a)=0$ and thus $a=0$, a contradiction.
Altogether, we have $e_0=0$.
Similarly one shows that $\bigvee_{t\in\RR}e_t=1$.

It is now straightforward to verify that the map $\RR\to\frakA$ given by $t\mapsto e_t$, is order-preserving and satisfies $e_t=\bigwedge_{s>t}e_{s}$ for all $t\in\RR$.
We also have $\bigwedge_{t\in\RR}e_t=0$ and $\bigvee_{t\in\RR}e_t=1$.
Thus, we obtain a unique element $\tfrac{d\sigma}{d\mu}\in L_{\mathbb{R}}^0(\frakA)_+$ such that $\level{\tfrac{d\sigma}{d\mu}\leq t}=e_t$ for $t\in(0,\infty)$.
We have $\level{\tfrac{d\sigma}{d\mu}>0}=1\setminus e_0=1$.

To verify \eqref{prp:Lamperti:RNderivative_levels:eq}, let $0< s\leq t< \infty$ and $a\leq\level{s\leq\frac{d\mu}{d\sigma}\leq t}$.
Then $a\leq\level{\tfrac{d\mu}{d\sigma}\leq t}=e_t$, which implies that $\sigma(a)\leq t\mu(a)$.
On the other hand, for every $s'<s$, we have
\[
a \leq \level{s\leq \tfrac{d\mu}{d\sigma}\leq t} \leq \level{s'<\tfrac{d\mu}{d\sigma}} = 1\setminus e_{s'}.
\]
By Claim~2, we deduce $s'\mu(a)\leq \sigma(a)$.
Since this holds for every $s'<s$, we conclude that $s\mu(a)\leq\sigma(a)$, as desired.
\end{proof}


The following is our desired Radon-Nikodym Theorem for localizable measure algebras.

\begin{thm}
\label{prp:Lamperti:RNderivative_integral}
Let $\frakA$ be a complete Boolean algebra, and let $\sigma,\mu\colon\frakA\to[0,\infty]$ be strictly positive, semi-finite, completely additive valuations.
Let $f\in L^1(\frakA,\sigma)$.
Then $f\tfrac{d\sigma}{d\mu}$ belongs to $L^1(\frakA,\mu)$ and we have
\begin{align}
\label{prp:Lamperti:RNderivative_integral:eq}
\int f\, d\sigma = \int f\frac{d\sigma}{d\mu}\, d\mu.
\end{align}
\end{thm}
\begin{proof}
We first prove the result for characteristic functions.
Let $a\in\frakA$ with $\sigma(a)<\infty$.
Using the monotone convergence theorem (see for example \cite[365C, p.116]{Fre-MsrThy3b}) at the first step, and using \eqref{prp:Lamperti:RNderivative_levels:eq} at the second step, we obtain
\begin{align*}
\int \charFct_a \frac{d\sigma}{d\mu}\, d\mu
&=\lim_{k\to\infty} \sum_{n=1}^\infty \frac{n}{2^k} \mu\left(a\cap\llevel{\frac{n}{2^k}< \frac{d\sigma}{d\mu} \leq \frac{n+1}{2^k}}\right) \\
&\leq \lim_{k\to\infty} \sum_{n=1}^\infty \frac{n}{2^k} \frac{2^k}{n} \sigma\left(a\cap\llevel{\frac{n}{2^k}< \frac{d\mu}{d\sigma} \leq \frac{n+1}{2^k}}\right) \\
&= \lim_{k\to\infty} \sigma\left(a\cap\llevel{\frac{1}{2^k}< \frac{d\mu}{d\sigma}}\right)= \sigma(a).
\end{align*}
Let $\varepsilon>0$.
Using \eqref{prp:Lamperti:RNderivative_levels:eq} at the second step,
we get
\begin{align*}
\int \sum_{n=1}^\infty (n+1)\varepsilon \charFct_{a\cap\llevel{n\varepsilon<\frac{d\sigma}{d\mu} \leq (n+1)\varepsilon}}\, d\mu
&= \sum_{n=1}^\infty (n+1)\varepsilon \mu\left(a\cap\llevel{n\varepsilon<\frac{d\sigma}{d\mu} \leq (n+1)\varepsilon}\right) \\
&\leq  \frac{n+1}{n}\varepsilon \sum_{n=1}^\infty  \sigma\left(a\cap\llevel{n\varepsilon<\frac{d\sigma}{d\mu} \leq (n+1)\varepsilon}\right) \\
&\leq 2\sigma(a)<\infty.
\end{align*}
We may therefore apply the dominated convergence theorem at the first step, and \eqref{prp:Lamperti:RNderivative_levels:eq} at the second step to deduce that
\begin{align*}
\int \charFct_{a\cap\llevel{\varepsilon< \frac{d\sigma}{d\mu}}} \frac{d\sigma}{d\mu}\, d\mu
&=\lim_{k\to\infty} \sum_{n=2^k}^\infty \frac{n+1}{2^k}\varepsilon \mu\left(a\cap\llevel{\frac{n}{2^k}\varepsilon < \frac{d\sigma}{d\mu} \leq \frac{n+1}{2^k}\varepsilon}\right) \\
&\geq \lim_{k\to\infty} \sum_{n=2^k}^\infty \frac{n+1}{2^k}\varepsilon \frac{2^k}{(n+1)\varepsilon} \sigma\left(a\cap\llevel{\frac{n}{2^k}\varepsilon< \frac{d\mu}{d\sigma} \leq \frac{n+1}{2^k}\varepsilon}\right) \\
&= \sigma\left(a\cap\llevel{\varepsilon < \frac{d\mu}{d\sigma}}\right).
\end{align*}
It follows that
\[
\int \charFct_a \frac{d\sigma}{d\mu}\, d\mu
= \lim_{\varepsilon\to 0} \int \charFct_{a\cap\llevel{\varepsilon< \frac{d\sigma}{d\mu}}} \frac{d\sigma}{d\mu}\, d\mu
\geq \lim_{\varepsilon\to 0} \sigma\left(a\cap\llevel{\varepsilon < \frac{d\mu}{d\sigma}}\right)
= \sigma(a).
\]
Together we have $\int \charFct_a \frac{d\sigma}{d\mu}\, d\mu = \sigma(a) = \int \charFct_a \, d\sigma$.

By linearity of the integral, we obtain $\int u \frac{d\sigma}{d\mu}\, d\mu = \int u \, d\sigma$ for every simple function $u$ in $L^1(\frakA,\sigma)$.
Since such functions are norm-dense in $L^1(\frakA,\sigma)$ (see for example \cite[Lemma~365F, p.118]{Fre-MsrThy3b}), and since the integral is norm-continuous, we obtain the same formula for every $f\in L^1(\frakA,\sigma)$.
\end{proof}

Using \autoref{prp:Lamperti:RNderivative_levels}\eqref{prp:Lamperti:RNderivative_levels:eq} it is straightforward
to prove the following basic rules for Radon-Nikodym derivatives, whose proofs are left to the reader.

\begin{lma}
\label{prp:Lamperti:RNderivative_formulas}
Let $\frakA$ be a complete Boolean algebra, and let $\sigma,\mu,\rho\colon\frakA\to[0,\infty]$ be strictly positive, semi-finite, completely additive valuations.
Then
\begin{align}
\label{prp:Lamperti:RNderivative_formulas:eqProd}
\frac{d\mu}{d\sigma} \frac{d\sigma}{d\rho}
= \frac{d\mu}{d\rho}.
\end{align}

Let $\varphi\in\Aut(\frakA)$.
Then
\begin{align}
\label{prp:Lamperti:RNderivative_formulas:eqPush}
\varphi\circ\frac{d\mu}{d\sigma} = \frac{d(\mu\circ\varphi^{-1})}{d(\sigma\circ\varphi^{-1})}.
\end{align}
\end{lma}
\section{The Banach-Lamperti Theorem for localizable measure algebras}
\label{sec:Lamperti}

In \cite{Lam58IsoLp}, Lamperti gave a description of the linear isometries of the $L^p$-space of a $\sigma$-finite measure space, for $p\in[1,\infty)$ with $p\neq 2$.
Since this result had been earlier announced (without proof) by Banach for the unit interval with
the Lebesgue measure, it is usually referred to as the ``Banach-Lamperti Theorem''.
In this section, we generalize their result by characterizing the surjective, linear isometries
on the $L^p$-space of a localizable measure algebra; see \autoref{prp:Lamperti:MainResult}.
The generalization from $\sigma$-finite spaces to localizable ones is not a minor one, and will allow us in the next sections to deal with locally compact groups
that are not $\sigma$-compact.

Besides wanting to avoid working only with $\sigma$-compact groups, another reason for presenting a rigorous proof of \autoref{prp:Lamperti:MainResult} is
that Lamperti's description (and its proof) is not entirely correct, as is pointed out in \cite{CamFauGar79IsomLpCopiesLpShift} and \cite{GioPes07ExtrAmenGps}.

\begin{ntn}
Let $\frakA$ be a Boolean algebra.
We let $\Aut(\frakA)$ denote the group of Boolean automorphisms of $\frakA$.
Further, we use $\mathcal{U}(L^\infty(\frakA))$ to denote the $\TT$-valued functions on $\frakA$, that is,
$\mathcal{U}(L^\infty(\frakA))
= \big\{ f\in L^0(\frakA) \colon  |f| \equiv 1 \big\}$.
If $(\frakA,\mu)$ is a measure algebra, and $p\in[1,\infty)$, then we use $\Isom(L^p(\mu))$ to denote the group of surjective, linear isometries of the Banach space $L^p(\mu)$.
\end{ntn}

Note that $\mathcal{U}(L^\infty(\frakA))$ is a group under multiplication.
The following result is straightforward to verify.

\begin{lma}
\label{prp:Lamperti:IsomLpMult}
Let $(\frakA,\mu)$ be a measure algebra, let $p\in[1,\infty)$, and let $f\in \mathcal{U}(L^\infty(\frakA))$.
Then the multiplication map $m_f\colon L^p(\mu)\to L^p(\mu)$ given by $m_f(\xi)=f\xi$ for all $\xi\in L^p(\mu)$ is a surjective, linear isometry.
Moreover, given $g\in \mathcal{U}(L^\infty(\frakA))$, we have $m_fm_g=m_{fg}$.
\end{lma}


\begin{lma}
\label{prp:Lamperti:IsomLpu_pi}
Let $(\frakA,\mu)$ be a localizable measure algebra, let $p\in[1,\infty)$, and let $\varphi\in\Aut(\frakA)$.
Then the linear map $u_\varphi\colon L^p(\mu)\to L^p(\mu)$ given by
\begin{align}
\label{prp:Lamperti:IsomLpu_pi:eq}
u_\varphi(\xi)=(\varphi\circ\xi)\cdot\left(\tfrac{d(\mu\circ\varphi^{-1})}{d\mu}\right)^{1/p},
\end{align}
for all $\xi\in L^p(\mu)$, is a surjective isometry.
Moreover, given $\psi\in\Aut(\frakA)$, we have $u_\psi\circ u_\varphi=u_{\psi\circ\varphi}$.
\end{lma}
\begin{proof}
Let $f\in L_{\mathbb{R}}^1(\mu)_+$.
Then
\begin{align*}
\int f d\mu
&= \int_{0}^\infty \mu\big( \level{f> t} \big)dt\\
&= \int_{0}^\infty (\mu\circ\varphi^{-1}\circ\varphi)\big( \level{f> t} \big)dt \\
&= \int_{0}^\infty (\mu\circ\varphi^{-1})\big( \level{ (\varphi\circ f) > t } \big)dt \\
&= \int \big( \varphi\circ f \big)\, d(\mu\circ\varphi^{-1}).
\end{align*}
It follows that the formula $\int f d\mu = \int \big( \varphi\circ f \big)\, d(\mu\circ\varphi^{-1})$ holds for arbitrary $f\in L^1(\mu)$.

Note that both $\mu$ and $\mu\circ\varphi^{-1}$ are strictly positive, semi-finite, completely additive valuations on $\frakA$.
We may therefore apply \autoref{prp:Lamperti:RNderivative_integral} and obtain that
\[
\int \big( \varphi\circ f \big)\, d(\mu\circ\varphi^{-1})
= \int \big( \varphi\circ f \big) \frac{d(\mu\circ\varphi^{-1})}{d\mu}\, d\mu.
\]

To show that $u_\varphi$ is well-defined and isometric, let $\xi\in L^p(\mu)$.
We have
\[
\left|(\varphi\circ\xi) \left(\frac{d(\mu\circ\varphi^{-1})}{d\mu}\right)^{1/p}\right|^p
= |(\varphi\circ\xi)|^p \frac{d(\mu\circ\varphi^{-1})}{d\mu}
\]
and therefore
\begin{align*}
\|u_\varphi(\xi)\|_p^p
&= \int \left|(\varphi\circ\xi) \left(\frac{d(\mu\circ\varphi^{-1})}{d\mu}\right)^{1/p}\right|^p\, d\mu
= \int |(\varphi\circ\xi)|^p \frac{d(\mu\circ\varphi^{-1})}{d\mu}\, d\mu \\
&= \int |(\varphi\circ\xi)|^p\, d(\mu\circ\varphi^{-1})
= \int |\xi|^p\, d\mu
= \|\xi\|_p^p.
\end{align*}
It follows that $u_\varphi\colon L^p(\mu)\to L^p(\mu)$ is a linear, isometric map.

To verify that  $u_\psi\circ u_\varphi=u_{\psi\circ\varphi}$, let $\xi\in L^p(\mu)$.
Using \autoref{prp:Lamperti:RNderivative_formulas} at the third and fourth steps, we obtain
\begin{align*}
(u_\psi\circ u_\varphi)(\xi)
&= u_\psi\left( (\varphi\circ\xi) \cdot \left(\frac{d(\mu\circ\varphi^{-1})}{d\mu}\right)^{1/p} \right) \\
&= (\psi\circ\varphi\circ\xi) \cdot \psi\circ\left(\frac{d(\mu\circ\varphi^{-1})}{d\mu}\right)^{1/p} \cdot \left(\frac{d(\mu\circ\psi^{-1})}{d\mu}\right)^{1/p} \\
&\stackrel{\eqref{prp:Lamperti:RNderivative_formulas:eqPush}}{=} (\psi\circ\varphi\circ\xi) \cdot \left( \frac{d(\mu\circ\varphi^{-1}\circ\psi^{-1})}{d(\mu\circ\psi^{-1})} \cdot \frac{d(\mu\circ\psi^{-1})}{d\mu}\right)^{1/p} \\
&\stackrel{\eqref{prp:Lamperti:RNderivative_formulas:eqProd}}{=} (\psi\circ\varphi\circ\xi) \cdot \left( \frac{d(\mu\circ\varphi^{-1}\circ\psi^{-1})}{d\mu}\right)^{1/p}
= u_{\psi\circ\varphi}(\xi).
\end{align*}

It is easy to see that $u_{\id}$ is the identity. Since $u_{\varphi^{-1}}\circ u_\varphi = u_{\id}$, it follows that $u_\varphi$ is surjective.
\end{proof}

Let $(\frakA,\mu)$ be a localizable measure algebra, and let $p\in[1,\infty)$.
By Lemmas~\ref{prp:Lamperti:IsomLpu_pi} and~\ref{prp:Lamperti:IsomLpMult}, every automorphism of $\frakA$, and 
every $\TT$-valued function on $\frakA$, induces surjective, linear isometries of $L^p(\mu)$.
Moreover, the resulting maps
$u\colon \Aut(\frakA)\to\Isom(L^p(\mu))$ and $m\colon \mathcal{U}(L^\infty(\frakA))\to\Isom(L^p(\mu))$ are group homomorphisms.
The next result clarifies the interplay between the two types of isometries of $L^p(\mu)$.

\begin{lma}
\label{prp:Lamperti:IsomLpInterplay}
Let $(\frakA,\mu)$ be a localizable measure algebra, let $\varphi\in\Aut(\frakA)$, and let $f\in \mathcal{U}(L^\infty(\frakA))$.
Then $u_\varphi m_f u_{\varphi^{-1}} = m_{\varphi\circ f}$.
\end{lma}
\begin{proof}
Given $\xi\in L^p(\mu)$, we have
\begin{align*}
u_\varphi(m_u(\xi))
&=(\varphi\circ(f\cdot\xi))\cdot\left(\tfrac{d(\mu\circ\varphi^{-1})}{d\mu}\right)^{1/p}
= (\varphi\circ f)\cdot (\varphi\circ\xi)\cdot\left(\tfrac{d(\mu\circ\varphi^{-1})}{d\mu}\right)^{1/p} \\
&= m_{\varphi\circ f}(u_\varphi(\xi)).
\end{align*}
It follows that $u_\varphi m_u=m_{\varphi\circ f}u_\varphi$.
\end{proof}

We let $\Aut(\frakA)$ act on $\mathcal{U}(L^\infty(\frakA))$ by $\pi\cdot f = \pi\circ f$.
It follows from \autoref{prp:Lamperti:IsomLpInterplay} that there is a group homomorphism 
$\mathcal{U}(L^\infty(\mu)) \rtimes \Aut(\frakA) \to \Isom(L^p(\mu))$.
The Banach-Lamperti Theorem for localizable measure algebras (\autoref{prp:Lamperti:MainResult}) 
implies that this map is surjective for $p\neq 2$.

\begin{rmk}\label{rmk:Clark}
In \cite{Cla36UnifCvxSpaces}, Clarkson proved that if $\mu$ denotes the Lebesgue measure on $[0,1]$ or the counting measure on $\NN$, then for any $\xi,\eta\in L^p(\mu)$, one has
\begin{align*}
\|\xi+\eta\|_p^p + \|\xi-\eta\|_p^p \geq 2(\|\xi\|_p^p+\|\eta\|_p^p),\quad \text{ if } 2\leq p<\infty \\
\|\xi+\eta\|_p^p + \|\xi-\eta\|_p^p \leq 2(\|\xi\|_p^p+\|\eta\|_p^p),\quad \text{ if } 1<p\leq 2.
\end{align*}
It is by now a folklore result that Clarkson's inequalities hold for arbitrary measure spaces, and thus for arbitrary measure algebras.
Moreover, for $p\neq 2$, one has equality in the above inequalities if and only if the product $\xi\cdot\eta$ is zero (almost everyhwere).
\end{rmk}

\begin{lma}
\label{prp:Lamperti:disjointPres}
Let $(\frakA,\mu)$ be a measure algebra, let $p\in (1,\infty)$ with $p\neq 2$, and let $T\colon L^p(\mu)\to L^p(\mu)$ be a linear isometry.
Then $T$ is disjointness preserving, that is, we have $T(\xi)\cdot T(\eta)=0$ for all $\xi,\eta\in L^p(\mu)$ with $\xi\cdot\eta=0$.
\end{lma}
\begin{proof}
Let $\xi,\eta\in L^p(\mu)$ satisfy $\xi\cdot\eta=0$.
Using that $T$ is linear and isometric at the first and last step, and using that $\xi\cdot\eta=0$ at the second step, we deduce that
\begin{align*}
\|T(\xi)+T(\eta)\|_p^p+\|T(\xi)-T(\eta)\|_p^p
& = \|\xi+\eta\|_p^p+\|\xi-\eta\|_p^p \\
&= 2\|\xi\|_p^p + 2\|\eta\|_p^p \\
&= 2\|T(\xi)\|_p^p + 2\|T(\eta)\|_p^p.
\end{align*}
We thus have equality in Clarkson's inequality for $T(\xi)$ and $T(\eta)$.
As explained in \autoref{rmk:Clark}, using that $p\neq 2$, we obtain that $T(\xi)\cdot T(\eta)=0$.
\end{proof}

For $z\in \mathbb{C}$ and $\varepsilon>0$, we denote by $B_\varepsilon(z)$
the open ball of radius $\varepsilon$ centered at $z$. 

\begin{thm}
\label{prp:Lamperti:MainResult}
Let $(\frakA,\mu)$ be a localizable measure algebra, let $p\in (1,\infty)$ with $p\neq 2$, and let $T\colon L^p(\mu)\to L^p(\mu)$ be a surjective, linear isometry.
Then there exist a unique Boolean automorphism $\varphi\colon\frakA\to\frakA$ and a unique function $f\in\mathcal{U}(L^\infty(\mu))$ such that $T=m_fu_\varphi$. 
It follows that
\[
\mathcal{U}(L^\infty(\mu)) \rtimes \Aut(\frakA) \cong \Isom(L^p(\mu)).
\]

Moreover, given $f,g\in\mathcal{U}(L^\infty(\mu))$ and $\varphi,\psi\in\Aut(\frakA)$, we have
\begin{align}
\label{prp:Lamperti:MainResult:eqNorm}
\| m_f u_\varphi - m_g u_\psi \|
= \max\{ \|f-g\|_\infty, 2\delta_{\varphi,\psi} \}.
\end{align}
\end{thm}
\begin{proof}
Given $\xi\in L^p(\mu)$, we set $\supp(\xi)=\level{|\xi|>0}\in\frakA$.
We define $\varphi\colon\frakA\to\frakA$ by
\[
\varphi(a)= \bigvee \big\{ \supp(T(\xi)) \colon  \xi\in L^p(\mu) \text{ with } \supp(\xi)\leq a \big\},
\]
for $a\in\frakA$.
Since $\mu$ is semi-finite, we have
\[
a= \bigvee \big\{ \supp(\xi) \colon  \xi\in L^p(\mu) \text{ with } \supp(\xi)\leq a \big\},
\]
for all $a\in\frakA$.
Since $T$ is bijective, it follows that $\varphi$ is bijective, with inverse given by
\[
\varphi^{-1}(a) = \bigvee \big\{ \supp(T^{-1}(\xi)) \colon  \xi\in L^p(\mu) \text{ with } \supp(\xi)\leq a \big\},
\]
for $a\in\frakA$.
It follows that $\varphi$ and $\varphi^{-1}$ are order-preserving.
It is also clear that $\varphi(0)=0$ and $\varphi(1)=1$.
To show that $\varphi$ is a Boolean isomorphism, it remains to show that $\varphi$ is additive on orthogonal elements.
So let $a,b\in\frakA$ with $a\wedge b=0$.
Let $\xi,\eta\in L^p(\mu)$ satisfy $\supp(\xi)\leq a$ and $\supp(\eta)\leq b$.
Then $\xi\cdot\eta=0$, and thus $T(\xi)\cdot T(\eta)=0$ by \autoref{prp:Lamperti:disjointPres}.
It follows that $\supp(T(\xi))$ and $\supp(T(\eta))$ are orthogonal.
Applying the definition of $\varphi$, it follows that $\varphi(a)$ and $\varphi(b)$ are orthogonal.
Since $\varphi$ is order-preserving, we deduce that $\varphi(a+b)\geq\varphi(a)+\varphi(b)$.
The same argument applies to $\varphi^{-1}$, whence
\[
\varphi(a)+\varphi(b)
=\varphi\big( \varphi^{-1}(\varphi(a)+\varphi(b)) \big)
\geq \varphi\big( \varphi^{-1}(\varphi(a)) + \varphi^{-1}(\varphi(b)) ) \big)
= \varphi(a+b).
\]
Thus $\varphi$ is a Boolean automorphism.

Set $S=T\circ u_\varphi^{-1} \in\Isom(L^p(\mu))$.
By construction, we have $\supp(S(\xi))=\supp(\xi)$ for every $\xi\in L^p(\mu)$.
Let $a\in\frakA$ with $\mu(a)<\infty$.

\textbf{Claim:} \emph{$S(\charFct_a)$ takes values in $\TT$ on $a$, that is, $\level{S(\charFct_a)\in\TT}=a$.}
To prove the claim, assume that it does not hold.
Choose $z\in\CC$ and $\varepsilon>0$ with $||z|-1|>\varepsilon^{1/p}$ such that the element
$b=\level{S(\charFct_a)\in B_\varepsilon(z)}$ is nonzero.
Then $\|\charFct_b\|_p=\mu(b)^{1/p}$.
On the other hand, we have $\|S(\charFct_b)-z\charFct_b\|_\infty\leq\varepsilon$ and therefore
\begin{align*}
\big| \|S(\charFct_b)\|_p - \|z\charFct_b\|_p \big|
\leq \|S(\charFct_b)-z\charFct_b\|_p
\leq \varepsilon^{1/p}\mu(b)^{1/p},
\end{align*}
while
\[
|\|z\charFct_b\|_p-\|\charFct_b\|_p|=||z|-1|\mu(b)^{1/p}>\varepsilon^{1/p}\mu(b)^{1/p},
\]
which together imply that $ \|S(\charFct_b)\|_p \neq  \|\charFct_b\|_p$.
This contradicts that $S$ is isometric, thus proving the claim.

Let $M\subseteq\CC$ be a measurable subset. 
It follows from additivity of $S$ that 
\[\level{S(\charFct_b)\in M}=b\wedge\level{S(\charFct_a)\in M}\] 
whenever $a,b\in\frakA$ satisfy $b\leq a$.
Thus, there is a unique function $f\in L^0(\frakA)$ with
\[
\level{f\in M} = \bigvee_{\mu(a)<\infty} \level{S(\charFct_a)\in M},
\]
for a measurable subset $M\subseteq\CC$.
Then $\level{f\in \TT}=1$ and thus $f\in\mathcal{U}(L^\infty(\frakA))$.
Further, for each $a\in\frakA$ with $\mu(a)<\infty$ we have $S(\charFct_a)=f\cdot \charFct_a$.
Using that $S$ is linear we obtain that $S$ is given by multiplication with $f$ for every simple function in $L^p(\mu)$.
Since $S$ is norm-continuous (even isometric) and since simple functions are norm-dense in $L^p(\mu)$, we deduce that $S=m_f$.
We omit the details.

It follows that $T\circ u_\varphi^{-1}=m_f$ and thus $T=m_f\circ u_\varphi$, as desired.

To verify the formula \eqref{prp:Lamperti:MainResult:eqNorm}, let $f,g\in\mathcal{U}(L^\infty(\mu))$, and let $\varphi,\psi\in\Aut(\frakA)$.
We first assume that $\varphi\neq\psi$.
Choose a nonzero $a\in\frakA$ with $\varphi(a)\neq\psi(a)$.
If $\varphi(a)\setminus\psi(a)\neq 0$, set $b=\varphi^{-1}(\varphi(a)\setminus\psi(a))$.
Otherwise, if $\psi(a)\setminus\varphi(a)\neq 0$, set $b=\psi^{-1}(\psi(a)\setminus\varphi(a))$.
In both cases, we have $b\neq 0$ and $\varphi(b)\wedge\psi(b)=0$.
Using that $(\frakA,\mu)$ is localizable, choose $c\in\frakA$ with $0\neq c\leq b$ and $\mu(c)<\infty$.
Set $\xi=\tfrac{1}{\mu(c)^{1/p}}\charFct_c$, which belongs to $L^p(\mu)$ with $\|\xi\|_p=1$.
By construction, $(m_f u_\varphi)(\xi)$ and $(m_g u_\psi)(\xi)$ have disjoint supports, whence
\[
\| m_f u_\varphi - m_g u_\psi \|
\geq \| ( m_f u_\varphi - m_g u_\psi)(\xi) \|_p
= \| (m_f u_\varphi)(\xi) \|_p + \|(- m_g u_\psi)(\xi) \|_p = 2.
\]
Conversely, we have $\| m_f u_\varphi - m_g u_\psi \|\leq \| m_f u_\varphi\| + \| m_g u_\psi \|=2$.

We now assume that $\varphi=\psi$.
Then $\|m_f u_\varphi - m_g u_\psi \|= \|m_f-m_g\|$.
The operator $m_f-m_g$ is given by multiplication by $f-g$, and it is standard to verify that this operator has norm
$\|f-g\|_\infty\leq 2$, which verifies \eqref{prp:Lamperti:MainResult:eqNorm}.
\end{proof}

\section{The unitary group of \texorpdfstring{$p$}{p}-convolution algebras}
\label{sec:UnitGpConvAlg}
\subsection{Convolution algebras}

Let $G$ be a \lcg, and fix a left Haar measure $\mu$ on $G$.
We let $M^1(G)$ denote the space of complex-valued Radon measures on $G$ of bounded variation.
Under convolution, $M^1(G)$ becomes a unital Banach algebra. 

For $p\in [1,\infty)$, we let $\lambda_p$ and $\rho_p$ denote the left and right regular representation of $G$ on $L^p(G)$, respectively, which are given by:
\[
\lambda_p(s)(\xi)(t) = \xi(s^{-1}t),\andSep
\rho_p(s)(\xi)(t) = \xi(ts),
\]
for $\xi\in L^p(G)$ and $s,t\in G$.

The integrated form of $\lambda_p$ is the contractive, nondegenerate representation of $L^1(G)$ on $L^p(G)$, also denoted by $\lambda_p$,
which is given by convolution on the left.
Similarly, $\mu\in M^1(G)$ acts by left convolution on $L^p(G)$.
This defines a unital, contractive representation of $M^1(G)$ on $L^p(G)$, also denoted by $\lambda_p$.

Let us recall some Banach algebras that are naturally associated with $G$:

\begin{dfn}
The Banach algebra $\PF_p(G)= \overline{\lambda_p(L^1(G))}^{\|\cdot\|}\subseteq \mathbb{B}(L^p(G))$ is the algebra of \emph{$p$-pseudofunctions} on $G$, and was introduced by Herz in \cite{Her73SynthSubgps}.
It has also been called the \emph{reduced group $L^p$-operator algebra} of $G$, and denoted by $F^p_\lambda(G)$, particularly
in the work of N.~Christopher Phillips.

We also consider the algebra of \emph{$p$-pseudomeasures}, denoted by $\PM_p(G)$, which is defined as the closure of $\PF_p(G)$ in the \weakStar-topology of $\Bdd(L^p(G))$.
Finally, the algebra of \emph{$p$-convolvers}, denoted by $\CV_p(G)$, is defined as follows:
\[
\CV_p(G) = \big\{ a\in\Bdd(L^p(G)) \colon a\rho_p(s) = \rho_p(s)a \text{ for all } s \in G \big\}.
\]
\end{dfn}

The algebras $\PM_p(G)$ and $\CV_p(G)$ are always unital, and $\PF_p(G)$ has a contractive, approximate identity.
If $G$ is nondiscrete, we will also consider the left-multiplier algebra of $\PF_p(G)$, which we denote by $M_l(\PF_p(G))$.
Since $\PF_p(G)$ has a contractive, approximate identity, and since $\PF_p(G)\subseteq\Bdd(L^p(G))$ is a nondegenerate subalgebra, it is
a standard result that $M_l(\PF_p(G))$ has a canonical, isometric representation as a unital subalgebra of $\mathbb{B}(L^p(G))$;
see for example \cite[Theorem~4.1]{GarThi17arX:ExtendingRepr} for a complete proof.

\begin{lma}
\label{prp:approxiIdCOnvId}
Let $G$ be a locally compact groyp, let $(f_j)_j$ be a contractive approximate identity in $L^1(G)$, and let $p\in(1,\infty)$.
Then the net $(\lambda_p(f_j))_j$ converges \weakStar{} to $1$ in $\Bdd(L^p(G))$.
\end{lma}
\begin{proof}
Let $p'$ be the dual H\"older exponent of $p$.
We identify $\Bdd(L^p(G))$ with the dual of $L^p(G)\tensMax L^{p'}(G)$.
Since the net $(\lambda_p(f_j))_j$ is bounded, it is enough to verify that $\langle \lambda_p(f_j),\omega\rangle \to \langle 1,\omega\rangle$ where $\omega$ is a simple tensor in $L^p(G)\tensMax L_{p'}(G)$.
Let $\xi\in L^p(G)$ and $\eta\in L_{p'}(G)$.
Then $\|f_j\ast\xi-\xi\|_p\to 0$, which implies that
\[
\langle \lambda_p(f_j),\xi\otimes\eta\rangle
= \langle f_j\ast\xi,\eta\rangle
\to \langle \xi,\eta\rangle
= \langle 1,\xi\otimes\eta\rangle,
\]
as desired.
\end{proof}

\begin{prp}
\label{prp:inclusions}
Let $G$ be a \lcg, and let $p\in(1,\infty)$.
Then there are natural (isometric) inclusions:
\[
\PF_p(G)
\subseteq M_l(\PF_p(G))
\subseteq \PM_p(G)
\subseteq \CV_p(G).
\]
\end{prp}
\begin{proof}
The last inclusion $\PM_p(G) \subseteq \CV_p(G)$ is well-known.
Let us verify that $\PF_p(G)\subseteq M_l(\PF_p(G))$. For this,
it is enough to verify $\lambda_p(L^1(G))\subseteq M_l(\PF_p(G))$.
Let $f\in L^1(G)$.
To verify that $\lambda_p(f)$ is a left multiplier for $\PF_p(G)$, let $b\in \PF_p(G)$.
We need to show that $\lambda_p(f)b\in \PF_p(G)$.
Choose a net $(b_j)_j$ in $L^1(G)$ such that $\lambda_p(b_j)\xrightarrow{\|\cdot\|}b$.
Since elements of $L^1(G)$ are left multipliers for $L^1(G)$, we have $fb_j\in L^1(G)$ and therefore $\lambda_p(f)\lambda_p(b_j)\in \PF_p(G)$, for each $j$.
Since $\lambda_p(f)\lambda_p(b_j)\xrightarrow{\|\cdot\|}\lambda_p(f)b$, we conclude that $\lambda_p(f)b\in \PF_p(G)$, as desired.

To verify that $M_l(\PF_p(G))\subseteq\PM_p(G)$, let $S\in M_l(\PF_p(G))$.
Let $(f_j)_j$ be a contractive approximate identity in $L^1(G)$.
By \autoref{prp:approxiIdCOnvId}, the net $(\lambda_p(f_j))_j$ converges \weakStar{} to $1$ in $\Bdd(L^p(G))$.
For each $j$, we have $\lambda_p(f_j)\in \PF_p(G)$, and therefore $S\lambda_p(f_j)\in \PF_p(G)$.
Since left multiplication by $S$ is \weakStar{} continuous on $\Bdd(L^p(G))$, we deduce that $S\lambda_p(f_j)\xrightarrow{w^*}S$.
Thus, $S$ belongs to the \weakStar-closure of $\PF_p(G)$, and hence to $\PM_p(G)$.
\end{proof}

\subsection{The unitary group of \texorpdfstring{$\CV_p(G))$}{CVp(G))}}
\label{sec:UCVp}

Throughout this section, we let $G$ denote a locally compact group.
We let $\Sigma$ denote the family of Haar measurable subsets of $G$, which includes all Borel subsets of $G$.
Further, $\mu\colon\Sigma\to[0,\infty]$ denotes a fixed left Haar measure for $G$.

We let $(\frakA,\mu)$ denote the measure algebra of $(G,\Sigma,\mu)$, that is, the Boolean algebra obtained by taking the quotient of $\Sigma$ by the ideal of null sets.
Given $M\in\Sigma$, we let $[M]$ denote the class of $M$ in $\frakA$.
It is well-known that the Haar measure on a locally compact group is localizable, which implies that $\frakA$ is a complete Boolean algebra and  $(\frakA,\mu)$ is localizable; see for example \cite[443A(a), p.290]{Fre-MsrThy4a}.

We let $\mathcal{T}$ denote the collection of open subsets of $G$, and we set
\[
\frakA_o = \big\{ [U] \colon  U\in\mathcal{T} \big\}, \andSep
\frakA_c= \big\{[K]\colon K\subseteq G \text{ compact} \big\}.
\]

The canonical quotient map $\Sigma\to\frakA$ preserves the order, finite infima and finite suprema, but not arbitrary suprema or infima (whenever they exist in $\Sigma$).
However, the restriction $\mathcal{T}\to\frakA_o$ does preserve arbitrary suprema;
see \cite[Theorem~414A, p.50]{Fre-MsrThy4a}.
It follows that $\frakA_o$ is closed under arbitrary suprema in $\frakA$.

Let $s\in G$.
The maps $\Sigma\to\Sigma$ given by $M\mapsto sM$ and $M\mapsto Ms$ induce maps on $\frakA$, and we denote them by $l_s$ and $r_s$, respectively.
It is straightforward to verify that $l_s,r_s\in\Aut(\frakA)$.
Since $l_{st}=l_s\circ l_t$ and $r_{st}=r_t\circ r_s$ for all $s,t\in G$, we obtain a left action $l\colon G\to\mathrm{Aut}(\frakA)$ and a right action $r\colon G\to\mathrm{Aut}(\frakA)$ of $G$ on $\frakA$;
for details we refer to \cite[Theorem~443C, p.291]{Fre-MsrThy4a}.

Given an arbitrary subset $Z\subseteq G$, we define the \emph{measurable cover} of $Z$ as
\[
[Z]^*=\bigwedge\big\{ [C]\colon C\in\Sigma, Z\subseteq C \big\} \in \frakA.
\]

\begin{prp}
\label{prp:UCVp:supTranslates}
Let $G$ be a locally compact group, and let $T,M\subseteq G$ with $M$ measurable.
Then there exists a measurable subset $M'\subseteq M$ with $[M']=[M]$ satisfying the following properties:
\begin{itemize}
 \item For every measurable subset $Y\subseteq G$ such that $tM'\setminus Y$ is negligible for every $t\in T$, the set $Y\setminus TM'$ is negligible.
 \item For every measurable subset $Y\subseteq G$ such that $M't\setminus Y$ is negligible for every $t\in T$, the set $Y\setminus M'T$ is negligible.
\end{itemize}
This means that the following identities hold in the complete Boolean algebra $\frakA$:
\[
\bigvee_{t\in T}l_t([M]) = [TM']^*, \andSep
\bigvee_{t\in T}r_t([M]) = [M'T]^*.
\]
\end{prp}
\begin{proof}
We first explain how to reduce the proof to the case that $\mu(M)$ is finite.
Using that $\mu$ is strictly localizable (see \cite[443A(a), p.290]{Fre-MsrThy4a}), we choose a partition of $G$ in the sense of \cite[Definition~211E, p.12]{Fre-MsrThy2}, that is, a family $(B_j)_j$ of pairwise disjoint, measurable sets with $\mu(B_j)<\infty$ that cover $G$, such that a subset $Z\subseteq G$ is measurable if and only if $Z\cap B_j$ is measurable for every $j\in J$, and such that $\mu(Z)=\sum_j\mu(Z\cap B_j)$ for every $Z\in\Sigma$.

Set $M_j=M\cap B_j$.
Assume that for each $j$, we can find a measurable subset $M_j'\subseteq M_j$ with $[M_j']=[M_j]$, and such that
\[
\bigvee_{t\in T}l_t([M_j]) = [TM_j']^*, \andSep
\bigvee_{t\in T}r_t([M_j]) = [M_j'T]^*.
\]
Set $M'=\bigcup_j M_j'$.
Then $M'\in\Sigma$, $M'\subseteq M$ and $[M']=[M]$.
Further, we have $[M]=\bigvee_j [M_j]$ in $\frakA$.
It is straightforward to verify that the measurable cover of $\bigcup_j TM_j'$ is given by $\bigvee_j[TM_j']^*$.
It follows that
\begin{align*}
\bigvee_{t\in T}l_t([M])
&= \bigvee_{t\in T}l_t(\bigvee_j[M_j])
= \bigvee_{t\in T}\bigvee_jl_t([M_j])
= \bigvee_j\bigvee_{t\in T}l_t([M_j]) \\
&= \bigvee_j [TM_j']^*
= [\bigcup_j TM_j']^*
= [TM']^*.
\end{align*}
Analogously, we obtain that $\bigvee_{t\in T}r_t([M])=[M'T]^*$.

It is therefore enough to assume, as we do from now on, that $\mu(M)<\infty$.
Choose $Y,Z\in\Sigma$ with $[Y]=\bigvee_{t\in T}l_t([M])$ and $[Z]=\bigvee_{t\in T}r_t([M])$.
Let $(h_j)_{j\in J}$ be an approximate identity for the convolution algebra $L^1(G)$.
Then $\lim_{j\in J}h_j\ast\charFct_Z = \charFct_Z$ and
\[
\lim_{j\in J}\charFct_M\ast h_j
= \lim_{j\in J}h_j\ast\charFct_M = \charFct_M, \ \
\lim_{j\in J}\charFct_Y\ast h_j = \charFct_Y
\]
in $L^1(G)$.
Using that $L^1(G)$ is metrizable, we choose a subsequence $(j_n)_{n\in\NN}$ of $J$ such that $\lim_{n\in\NN}h_{j_n}\ast\charFct_Z = \charFct_Z$ and
\[
\lim_{n\in\NN}\charFct_M\ast h_{j_n}
= \lim_{n\in\NN}h_{j_n}\ast\charFct_M = \charFct_M, \ \
\lim_{n\in\NN}\charFct_Y\ast h_{j_n} = \charFct_Y,
\]
in $L^1(G)$.
Every $L^1$-convergent sequence admits a subsequence that converges pointwise almost everywhere.
We first obtain a subsequence $(n(k))_k$ such that $(\charFct_M\ast h_{j_{n(k)}})_k$ converges pointwise almost everywhere.
Applying the same result to the sequence $(h_{j_{n(k)}}\ast\charFct_M)_k$ and then to the sequences involving $\charFct_Y$ and $\charFct_Z$, and after reindexing,
we find a conull subset $G'\subseteq G$ and a sequence $(j_k)_{k\in\NN}$ in $J$ such that
\begin{align*}
\lim_{k\in\NN} (\charFct_M\ast h_{j_k})(x)
&= \lim_{k\in\NN} (h_{j_k}\ast\charFct_M)(x)
= \charFct_M(x), \\
\lim_{k\in\NN} (\charFct_Y\ast h_{j_k})(x)
= \charFct_Y(x), &\andSep
\lim_{k\in\NN} (h_{j_k}\ast\charFct_Z)(x)
= \charFct_Z(x),
\end{align*}
for every $x\in G'$.
Set $M'=M\cap G'$.
Then $M'\subseteq M$ and $[M']=[M]$.
For each $t\in T$, the set $tM'\setminus Y$ is a null set.
Using this at the second step, we deduce that
\[
1\geq (\charFct_{Y}\ast h_{j_k})(tx)
\geq (\charFct_{tM'}\ast h_{j_k})(tx)
= (\charFct_{M'}\ast h_{j_k})(x),
\]
for every $t\in T$,  $x\in M'$ and $k\in\NN$.
It follows that
\[
\lim_{k\in\NN}(\charFct_{Y}\ast h_{j_k})(y)=1,
\]
for every $y\in TM'$. 
On the other hand, we have $\lim_{k\in\NN} (\charFct_Y\ast h_{j_k})(y)=\charFct_Y(y)$ for $y\in G'$.
It follows that $TM'\cap G'\subseteq Y$.
Then
\[
[TM']^*
= [TM'\cap G']^*
\leq[Y]
=\bigvee_{t\in T}l_t([M]).
\]
On the other hand, we have
\[
l_t([M])=l_t([M'])=[tM']
= [tM']^*
\leq [TM']^*,
\]
and thus $\bigvee_{t\in T}l_t([M])\leq[TM']^*$.
Analogously, we get  $\bigvee_{t\in T}r_t([M])=[MT']^*$.
\end{proof}

\begin{cor}
\label{prp:UCVp:supOpenFrakA}
Let $a\in\frakA$, and let $V\subseteq G$ be open.
Then $\bigvee_{s\in V}l_s(a)$ and $\bigvee_{s\in V}r_s(a)$ belong to $\frakA_o$.
\end{cor}
\begin{proof}
Choose $M\in\Sigma$ with $a=[M]$.
Apply~\autoref{prp:UCVp:supTranslates} to obtain $M'\in\Sigma$ with $[M']=[M]=a$ and such that $\bigvee_{t\in T}l_t(a)=[VM']^*$ and $\bigvee_{t\in T}l_t(a)=[VM']^*$.
Then $VM'$ and $M'V$ are open and therefore measurable.
It follows that $\bigvee_{t\in T}l_t([M])=[VM']\in\frakA_o$ and $\bigvee_{t\in T}r_t([M])=[M'V]\in\frakA_o$.
\end{proof}

In preparation for \autoref{prp:UCVp:charLeftTrans}, we characterize when a measurable subset of a
locally compact group is equivalent to an open (or closed) set. This characterization is new and of
independent interest.

\begin{prp}
\label{prp:UCVp:charOpen}
Let $G$ be a locally compact group, let $\frakA$ be the associated complete Boolean algebra of Haar measurable sets modulo null sets, and let $a\in\frakA$.
Denote by $\mathcal{N}$ the family of open neighborhoods of the unit in $G$.
Then the following are equivalent:
\begin{enumerate}
\item
We have $a\in\frakA_o$. 
\item
The left action $l\colon G\to\mathrm{Aut}(\frakA)$ is lower semicontinuous at $a$, that is, 
$a=\bigvee_{V\in\mathcal{N}} \bigwedge_{s\in V}l_s(a)$.
\item
The right action $r\colon G\to\mathrm{Aut}(\frakA)$ is lower semicontinuous at $a$, that is, $a=\bigvee_{V\in\mathcal{N}} \bigwedge_{s\in V}r_s(a)$.
\end{enumerate}

Analogously, we have $a=[M]$ for some closed subset $M\subseteq G$ if and only if
$l\colon G\to\mathrm{Aut}(\frakA)$ (or equivalently  $r\colon G\to\mathrm{Aut}(\frakA)$) is upper semicontinuous at $a$.
\end{prp}
\begin{proof}
To show that~(1) implies~(2), choose an open set $U\subseteq G$ with $a=[U]$.
Using at the third step that the sets $\bigcup_{t\in V}t \left( \bigcap_{s\in V^{-1}V}sU \right)$ are open, we deduce that
\begin{align*}
\bigvee_{V\in\mathcal{N}} \bigwedge_{s\in V}l_s(a)
&= \bigvee_{V\in\mathcal{N}} \bigwedge_{s\in V}[sU]
\geq \bigvee_{V\in\mathcal{N}} \left[\bigcap_{s\in V}sU \right] \\
&\geq \bigvee_{V\in\mathcal{N}} \left[ \bigcup_{t\in V}t \left( \bigcap_{s\in V^{-1}V}sU \right) \right] \\
&= \left[ \bigcup_{V\in\mathcal{N}} \bigcup_{t\in V} t \left( \bigcap_{s\in V^{-1}V}sU \right) \right]
= [U] = a.
\end{align*}
The converse inequality $\bigvee_{V\in\mathcal{N}} \bigwedge_{s\in V}l_s(a)\leq a$ always holds.
The proof that~(1) implies~(3) is analogous.

To show that~(2) implies~(1), assume that $a=\bigvee_{V\in\mathcal{N}} \bigwedge_{s\in V} l_s(a)$.
For each $t\in V^{-1}$, we have $1\in tV$, which implies that
\[
l_{t}(\bigwedge_{s\in V}l_s(a))
= \bigwedge_{s\in V}l_{ts}(a)
\leq a.
\]
Passing to the supremum over $t\in V^{-1}$, we obtain
\[
\bigwedge_{s\in V}l_s(a) \leq \bigvee_{t\in V^{-1}}l_t(\bigwedge_{s\in V}l_s(a)) \leq a.
\]
It follows that
\[
a=\bigvee_{V\in\mathcal{N}} \bigwedge_{s\in V}l_s(a)
\leq \bigvee_{V\in\mathcal{N}} \bigvee_{t\in V^{-1}}l_t(\bigwedge_{s\in V}l_s(a)) \leq a.
\]
By \autoref{prp:UCVp:supOpenFrakA}, we have $\bigvee_{t\in V^{-1}}l_t(\bigwedge_{s\in V}l_s(a))\in\frakA_o$ for every $V\in\mathcal{N}$.
Since $\frakA_o$ is closed under arbitrary suprema in $\frakA$, it follows that $a\in\frakA_0$, as desired.
The proof that~(3) implies~(1) is analogous.

Finally, the characterization for closed sets is proved similarly, and we omit the details.
\end{proof}

Our next goal is to establish an \emph{algebraic} identification of $\TT\times G$ with the invertible isometries in $\CV_p(G)$.
For general locally compact groups, this takes some work, and we will accomplish this in \autoref{prp:UCVp:identifyUCVp}.
The biggest difficulty is dealing with the topology of $G$, and proving that the Boolean automorphism obtained from
the Banach-Lamperti theorem can be lifted to a continuous map on $G$.
This difficulty vanishes when $G$ is discrete, and a very simple proof of \autoref{prp:UCVp:identifyUCVp} can be given in that case.
For the convenience of the reader, and since the case of a discrete group contains the basic ideas of the general proof,
we outline the argument in this situation in \autoref{lma:DiscreteCase}. 

Recall that $\lambda_p$ and $\rho_p$ denote the left and right regular representation of $G$ on $L^p(G)$, respectively.


\begin{lma}\label{lma:DiscreteCase}
Let $G$ be a discrete group, let $p\in (1,\infty)$ with $p\neq 2$, and let $u\in\mathcal{U}(CV_p(G))$.
Then there exist $\gamma\in\TT$ and $t\in G$ such that $u=\gamma \lambda_p(t)$.
\end{lma}
\begin{proof}
Let $s\in G$.
We define the left and right translation maps $\Lt_s,\Rt_s\colon G\to G$ on $G$ by
$\Lt_s(x)=sx$ and $\Rt_s(x)=xs^{-1}$
for all $x\in G$. It is clear that $\lambda_p(s)=u_{\Lt_s}$ and 
$\rho_p(s)=u_{\Rt_{s}}$.

Now let $u\in \mathcal{U}(\CV_p(G))\subseteq \Isom(\ell^p(G))$.
For atomic $L^p$-spaces, the Banach-Lamperti Theorem takes the following simple form:
there exist a bijection $\varphi\colon G\to G$ and a map $h\colon G\to\TT$ such that
$u=m_hu_\varphi$.

Using the assumption that $u$ and $\rho_p(s)$ commute at the third step, and using \autoref{prp:Lamperti:IsomLpInterplay} at the last step, we get
\[
m_h u_{\varphi\circ \Rt_s}
= m_h u_\varphi u_{\Rt_s}
= u \rho_p(s)
= \rho_p(s) u
= u_{\Rt_s} m_h u_\varphi
= m_{h\circ\Rt_s} u_{\Rt_s\circ \varphi}.
\]
This implies that $h=h\circ\Rt_s$ and $\varphi\circ \Rt_s=\Rt_s\circ \varphi$.
Since this holds for all $s\in G$, we immediately deduce that $h$ is a constant function. Denote its unique value by $\gamma\in\TT$, and
set $t=\varphi(e)^{-1}$.
For $x\in G$, we compute
\[
\varphi(x)
= \varphi(ex)
= (\varphi\circ\Rt_{x^{-1}})(e)
= (\Rt_x\circ \varphi)(e)
= \varphi(e)x
= \Lt_t(x)
\]
which implies that $\varphi=\Lt_t$.
Thus, $u_\varphi=\lambda_p(t)$ and consequently $u = m_h u_\varphi = \gamma \lambda_p(t)$, as desired.
\end{proof}

\begin{lma}
\label{prp:UCVp:charLeftTrans}
Let $\pi\in\Aut(\frakA)$ such that $r_t\circ\pi=\pi\circ r_t$ for every $t\in G$.
Then there exists a unique $s\in G$ such that $\pi=l_s$.
\end{lma}
\begin{proof}
Recall that $\frakA_c$ denotes the equivalence classes of compact sets in $G$.
In the first two claims, we show that $\pi(\frakA_c)\subseteq\frakA_c$.
	
\textbf{Claim~1.}
\emph{We have $\pi(\frakA_o)\subseteq\frakA_o$.}
To prove the claim, let $a\in\frakA_o$.
Let $\mathcal{N}$ denote the family of open neighborhoods of the identity in $G$.
By \autoref{prp:UCVp:charOpen}, we have
\[
a=\bigvee_{V\in\mathcal{N}} \bigwedge_{t\in V}r_t(a).
\]
Using at the second step that $\pi$ is an order-isomorphism, and using at the third step that $\pi$ commutes with each $r_t$, we deduce that
\[
\pi(a)
= \pi\left( \bigvee_{V\in\mathcal{N}} \bigwedge_{t\in V}r_t(a) \right)
= \bigvee_{V\in\mathcal{N}} \bigwedge_{s\in V}\pi(r_t(a))
= \bigvee_{V\in\mathcal{N}} \bigwedge_{s\in V}r_t(\pi(a)).
\]
Applying \autoref{prp:UCVp:charOpen} again, we conclude that $\pi(a)\in\frakA_o$, proving the claim.

\textbf{Claim~2.}
\emph{Let $a\in\frakA$.
Then $a=[M]$ for some pre-compact subset $M\subseteq G$ if and only if there exists a compact subset $K\subseteq G$ and $b\neq 0$ such that
\[a\wedge\bigvee_{t\notin K}r_t(b)=0.\]}
To prove the forward implication, let $a=[M]$ with $M$ pre-compact.
Since $[M]\leq [\overline{M}]$, we may assume that $M$ is compact.
Let $V$ be any pre-compact neighborhood of the unit.
Set $K=\overline{VM}$, which is compact, and set $b=[V^{-1}]$.
Then
\[
\bigvee_{t\notin K}r_t(b)
= \bigvee_{t\notin K}[V^{-1}t]
= [\bigcup_{t\notin K} V^{-1}t]
= [V^{-1}K^c],
\]
and therefore
\[
a\wedge \bigvee_{t\notin K}r_t(b)
= [M]\wedge [V^{-1}K^c] = [M\cap V^{-1}K^c]=[\emptyset]=0.
\]
Conversely, assume that $a\wedge\bigvee_{t\notin K}r_t(b)=0$ for some compact subset $K\subseteq G$ and $b\neq 0$.
Applying \autoref{prp:UCVp:supTranslates} we obtain $B'\in\Sigma$ with $[B']=b$ such that $\bigvee_{t\notin K}r_t(b)=[B'K^c]^*$.
Pick $x\in B'$.
Then
\[
\bigvee_{t\notin K}r_t(b)=[B'K^c]^*\geq[xK^c],
\]
and therefore $a\leq [xK]$. This implies that $a$ has a pre-compact representative, and proves the claim.

It follows from Claim~1 that $\pi$ maps the family of equivalence classes of closed sets to itself.
Together with Claim~2 we deduce that $\pi(\frakA_c)\subseteq\frakA_c$.

We set $\frakA_c^\times=\frakA_c\setminus\{0\}$.
Given $a\in\frakA_c^\times$, we set
\[
L_a = \big\{ s\in G : l_s(a)\wedge\pi(a)\neq 0 \big\}.
\]
Further, we set
\[
L = \bigcap_{a\in\frakA_c^\times} \overline{L_a} \subseteq G.
\]
We are going to verify that $L$ contains exactly one element.

\textbf{Claim~3.}
\emph{Let $a\in\frakA$ with $a\neq 0$.
Then $\bigvee_{s\in G}l_s(a)=\bigvee_{s\in G}r_s(a)=1$.}
To prove the claim, apply \autoref{prp:UCVp:supTranslates} to obtain $M'\in\Sigma$ with $[M']=a$ such that
\[
\bigvee_{s\in G}l_s(a)=[GM']^*, \andSep
\bigvee_{s\in G}r_s(a)=[M'G]^*.
\]
Since $a\neq 0$, we have $M'\neq\emptyset$ and therefore $[GM']^*=[G]^*=1$ and $[MG']^*=[G]^*=1$.

\textbf{Claim~4.}
\emph{Let $a\in\frakA_c^\times$.
Then $L_a$ is nonempty and pre-compact.}
If $L_a$ were empty, then $l_s(a)\wedge\pi(a)=0$ for every $s\in G$.
Using claim~3 at the third step, we obtain
\[
0
= \bigvee_{s\in G} (l_s(a)\wedge\pi(a))
= \left(\bigvee_{s\in G} l_s(a)\right)\wedge\pi(a)
= \pi(a),
\]
which is a contradiction.
Thus, $L_a\neq\emptyset$.

To show that $L_a$ is pre-compact, choose compact subsets $K,C\subseteq G$ with $a=[K]$ and $\pi(a)=[C]$.
If $s\in L_a$, then $sK\cap C\neq\emptyset$, and therefore $s\in CK^{-1}$.
Hence $L_a\subseteq CK^{-1}$, which shows that $L_a$ is pre-compact.

\textbf{Claim~5.}
\emph{Let $a,b\in\frakA_c^\times$.
Then $L_a\cap L_b$ is nonempty.}
To prove the claim, let $s,t\in G$.
Then
\[
r_t\big( l_s(a)\wedge\pi(a) \big)
= r_t(l_s(a))\wedge r_t(\pi(a))
= l_s(r_t(a))\wedge \pi(r_t(a)),
\]
which implies that $L_{r_t(a)}=L_a$.
By Claim~3, we have $\bigvee_{t\in G}r_t(a)=1$.
It follows that there exists $t\in G$ such that the element $c=r_t(a)\cap b$ is nonzero.
Then $L_c\subseteq L_{r_t(a)}=L_a$, $L_c\subseteq L_b$, and $L_c\neq\emptyset$, which proves the claim.

It follows from Claims~4 and~5 that $\{\overline{L_a}:a\in\frakA_c^\times\}$ is a nested family of compact subsets of $G$, which implies that its intersection $L$ is nonempty.

\textbf{Claim~6.}
\emph{The set $L$ contains exactly one element.}
To prove the claim, assume that $s,t\in L$ with $s\neq t$.
Choose a neighborhood $V$ of the unit in $G$ such that $sV\cap tV=\emptyset$.
Choose subsets $W,K\subseteq G$ with $W$ open and $K$ compact such that $WK\subseteq V$ and the 
element $a=[K]$ is nonzero.
We have $s\in\overline{L_a}$, and therefore $sW\cap L_a\neq\emptyset$.
Choose $w\in W$ with $sw\in L_a$.
It follows that the element
\[
b=l_{sw}(a)\cap\pi(a).
\]
is nonzero. Then $\pi^{-1}(b)\neq 0$.
Using that the Haar measure is inner regular, choose $c\in\frakA_c^\times$ with $c\leq\pi^{-1}(b)$.
Then
\[
\pi(c)\leq b \leq l_{sw}(a)=[swK]\leq[sV].
\]
We have $t\in L_c$, and therefore $tW\cap L_c\neq\emptyset$.
Choose $w'\in W$ with $tw\in L_c$.
We have
\[
l_{tw'}(c) \leq l_{tw'}(a)=[tw'K]\leq[tV].
\]
Since $\pi(c)\leq[sV]$, we obtain
\[
l_{tw'}(c)\wedge\pi(c)
\leq [tV]\wedge[sV]=0,
\]
which contradicts that $sw'\in L_c$.

Thus, $L$ contains exactly one element.
Let $s\in G$ such that $L=\{s\}$.
Next, we show that $\pi=l_s$.

Let $a\in\frakA_c^\times$.
Recall that $\mathcal{N}$ denotes the family of open neighborhoods of the identity in $G$.
Let $V\in\mathcal{N}$.
We first show that $\pi(a)\leq \bigvee_{t\in Vs}l_t(a)$.
Assume otherwise.
Set $b=\pi(a)\setminus \bigvee_{t\in Vs}l_t(a)$. Then $b\neq 0$, which allows us to choose $c\in\frakA_c^\times$ with $c\leq a$ such that $\pi(c)\leq b$.
It follows that $Vs\cap L_c=\emptyset$, a contradiction.

Thus, $\pi(a)\leq \bigvee_{t\in Vs}l_t(a)$ for every  $V\in\mathcal{N}$.
Applying \autoref{prp:UCVp:charOpen} for $l_s(a)$ at the second step, we obtain
\[
\pi(a)
\leq \bigwedge_{v\in\mathcal{N}}\bigvee_{t\in V}l_t(l_s(a))
= l_s(a).
\]
It follows that $\pi(a)\leq l_s(a)$ for every $a\in\frakA_c$.

For general $a\in\frakA$, we have $a=\bigvee\{c\in\frakA_c : c\leq a\}$, and therefore
\begin{align*}
\pi(a)
&= \pi\left( \bigvee\{c\in\frakA_c \colon c\leq a\} \right)
= \bigvee\{\pi(c)\colon c\in\frakA_c, c\leq a\} \\
&\leq \bigvee\{l_s(c)\colon c\in\frakA_c, c\leq a\}
= l_s\left( \bigvee\{c\in\frakA_c \colon c\leq a\} \right)
= l_s(a).
\end{align*}
We also have $\pi(a^c)\leq l_s(a^c)$ and hence
\[
\pi(a)=\pi(a^c)^c \geq l_s(a^c)^c = l_s(a).
\]
Finally, we deduce that $\pi(a)=l_s(a)$ for every $a\in\frakA$.
It easily follows that $s\in G$ with this property is unique.
\end{proof}

Given a unital Banach algebra $A$, we let $\Unitary(A)$ denote the group of invertible isometries in $A$.

The following result has been proved in \cite[Theorem~1]{Par68IsoMultiplier}, relying
on Lamperti's theorem from \cite{Lam58IsoLp}.
However, as explained at the beginning of \autoref{sec:Lamperti}, Lamperti's formulation is not entirely
correct (and applies only to $\sigma$-finite spaces). For these reasons, we must provide a rigorous proof
of \autoref{prp:UCVp:identifyUCVp}, for arbitrary locally compact spaces, relying instead on \autoref{prp:Lamperti:MainResult}.

\begin{lma}
\label{prp:UCVp:identifyUCVp}
Let $G$ be a locally compact group, and let $p\in (1,\infty]$ with $p\neq 2$.
Let $u\in\mathcal{U}(\CV_p(G))$.
Then there exist unique $\gamma\in\TT$ and $t\in G$ such that $u=\gamma \lambda_p(t)$.

Further, given $\beta,\gamma\in\TT$ and $s,t\in G$ we have
\begin{align}
\label{prp:UCVp:identifyUCVp:eqNorm}
\| \beta \lambda_p(s) - \gamma \lambda_p(t) \| = \max\{|\beta-\gamma|,2\delta_{s,t}\}.
\end{align}
\end{lma}
\begin{proof}
Applying \autoref{prp:Lamperti:MainResult} to $u\in \Isom(L^p(\mu))$, we obtain a unique 
$\varphi\in\Aut(\frakA)$ and a unique $h\in\mathcal{U}(L^\infty(\frakA))$ such that $u=m_h u_\varphi$.

Let $s\in G$.
Recall that right multiplication by $s$ induces a Boolean automorphism $r_s\in\Aut(\frakA)$.
Note that $u_{r_s}=\lambda_p(s)\in \mathcal{U}(L^p(\frakA))$.
Since $\CV_p(G)$ is the commutator of the right regular representation, we obtain that
$u_{r_s} m_h u_\varphi = m_h u_\varphi u_{r_s}$.
Applying Lemmas~\ref{prp:Lamperti:IsomLpu_pi} and~\ref{prp:Lamperti:IsomLpInterplay}, we deduce that
\[
m_{r_s\circ h} u_{r_s\circ\varphi} = m_h u_{\varphi\circ r_r}.
\]
It follows that $r_s\circ h=h$ and $r_s\circ\varphi=\varphi\circ r_s$ for every $s\in G$.

Applying \autoref{prp:UCVp:charLeftTrans}, we obtain a unique $t\in G$ such that $\varphi=l_t$.
It is straightforward to verify that the identity $r_s\circ h=h$, valid for all $s\in G$, implies that $h$ is constant.
With $\gamma$ denoting the unique value of $h$, we have
$u=\gamma u_{l_t}=\gamma \lambda_p(t)$, as desired.
Finally, \eqref{prp:UCVp:identifyUCVp:eqNorm} follows from \eqref{prp:Lamperti:MainResult:eqNorm}.
\end{proof}

\begin{pgr}
\label{pgr:topologiesBddE}
Let $E$ be a reflexive Banach space.
We identify $\Bdd(E)$ with the dual of $E\tensMax E^*$.
Let us recall the definitions of the strong operator topology (SOT), the weak operator topology (WOT), and the \weakStar-topology (also called ultraweak topology) on $\Bdd(E)$.
Given a net $(a_j)_j$ in $\Bdd(E)$ and $a\in\Bdd(E)$, we have $a_j\xrightarrow{SOT}a$ if and only if $a_j\xi\xrightarrow{\|\cdot\|}a\xi$ for every $\xi\in E$.
Moreover, we have $a_j\xrightarrow{WOT}a$ if and only if $a_j\xi\xrightarrow{w}a\xi$ for every $\xi\in E$.
Since $E$ is reflexive, this is also equivalent to $a_j\xi\xrightarrow{w^*}a\xi$ for every $\xi\in E$.
Given $\xi\in E$ and $\eta\in E^*$, we consider the simple tensor $\xi\otimes\eta\in E\tensMax E^*$.
Under the duality of $\Bdd(E)$ with $E\tensMax E^*$, we have $\langle b,\xi\otimes\eta\rangle = \langle b(\xi),\eta\rangle$ for $b\in\Bdd(E)$.
It follows that $a_j\xrightarrow{WOT}a$ if and only if
\[
\langle a_j, \xi\otimes\eta\rangle \to \langle a, \xi\otimes\eta\rangle,
\]
for every $\xi\in E$ and $\eta\in E^*$.
Finally, we have $a_j\xrightarrow{w^*}a$ if and only if $\langle a_j,\omega\rangle \to \langle a,\omega\rangle$ for every $\omega\in E\tensMax E^*$.

It is easy to see that SOT-convergence implies WOT-convergence, and that \weakStar-convergence implies WOT-convergence.
In general, none of these implications can be reversed.
Further, SOT-convergence does not imply \weakStar-convergence, or vice versa.

However, it is well-known that the \weakStar-topology and the WOT agree on bounded sets of $\Bdd(E)$.
Further, the WOT and SOT agree on bounded subgroups of $\Gl(E)$;
see \cite[Theorem~2.8]{Meg01OpTopReflRepr}.
In particular, on $\Isom(E)$, the WOT, the SOT, and the \weakStar-topology agree.	
\end{pgr}

\begin{ntn}
\label{ntn:UCVp:pi0}
Let $A$ be a unital Banach algebra.
We let $\Unitary(A)$ denote the group of invertible isometries in $A$.
Further, $\Unitary(A)_0$ denotes the connected component of $\Unitary(A)$ in the norm topology that contains the unit of $A$.
Then $\Unitary(A)_0$ is a normal subgroup of $\Unitary(A)$ and we set
\[
\pi_0(\Unitary(A)) = \Unitary(A) / \Unitary(A)_0.
\]
If $A$ has a \weakStar-topology, then we write $\Unitary(A)_{w^*}$ for the group $\Unitary(A)$ equipped with the restriction of the \weakStar-topology.
We write $\pi_0(\Unitary(A))_{w^*}$ for the group $\pi_0(\Unitary(A))$ equipped with the quotient topology induced by $\Unitary(A)_{w^*}\to \pi_0(\Unitary(A))$.
\end{ntn}

\begin{rmk}
Let $A$ be a unital Banach algebra with a predual.
In general, multiplication in $A$ may not be continuous for the \weakStar-topology, and hence $\Unitary(A)_{w^*}$ may not be a topological group.

By definition, $A$ is a \emph{dual Banach algebra} if its predual makes multiplication separately \weakStar-continuous.
This notion was introduced by Runde in~\cite{Run02BookAmen} and has been extensively studied by Daws in~\cite{Daw07DualBAlgReprInj,Daw11BicommDualBAlg}.
If $A$ is a unital dual Banach algebra, then multiplication in $\Unitary(A)_{w^*}$ is separately continuous.
Using Daws' theorem of reflexive representability (\cite[Corollay~3.8]{Daw07DualBAlgReprInj}),
and using the observations from \autoref{pgr:topologiesBddE}, we will show in \autoref{prp:UnitGpTopGp} that $\Unitary(A)_{w^*}$ is a topological group.
\end{rmk}




The following is arguably the main result of this section, and 
it shows how one can recover $G$ (as a topological group) 
from $\CV_p(G)$ using the \weakStar-topology.

\begin{thm}
\label{prp:UCVp:TopUCVp}
Let $G$ be a locally compact group, and let $p\in(1,\infty)$ with $p\neq 2$.
Then the maps
\[
\Delta\colon\TT\times G\to\Unitary(\CV_p(G))_{w^*},\andSep
\Delta'\colon G\to \pi_0(\Unitary(\CV_p(G)))_{w^*},
\]
given by $\Delta(\gamma,s)=\gamma \lambda_p(s)$, and by $\Delta'(s)=[\lambda_p(s)]$, are isomorphisms of topological groups.
\end{thm}
\begin{proof}
We have the following diagram:
\[
\xymatrix@R-5pt{
\TT \times G \ar[r]^-{\Delta} \ar[d]_{(\gamma,s)\mapsto s}
& \Unitary(\CV_p(G))_{w^*} \ar[d]^{u\mapsto[u]} \\
G \ar[r]^-{\Delta'}
& { \pi_0(\Unitary(\CV_p(G)))_{w^*} }.
}
\]
It follows from \autoref{prp:UCVp:identifyUCVp} that $\Delta$ and $\Delta'$ are bijective group homomorphisms.
The downwards maps are topological quotient maps.
Therefore, it is enough to verify that $\Delta$ is a homeomorphism.

To verify this, let $(\gamma_j,s_j)_j$ be a net in $\TT\times G$, and let $(\gamma,s)\in\TT\times G$.
The \weakStar-topology on $\CV_p(G)$ is the restriction of the \weakStar-topology on $\Bdd(L^p(G))$.
Further, as noted in \autoref{pgr:topologiesBddE}, on $\Isom(L^p(G))$ the \weakStar-topology agrees with the strong operator topology (SOT).
Thus, we need to show that $(\gamma_j,s_j)\to(\gamma,s)$ in $\TT\times G$ if and only if $\gamma_j \lambda_p(s_j)\xrightarrow{SOT}\gamma \lambda_p(s)$.

Assume that $(\gamma_j,s_j)\to(\gamma,s)$, and let $\xi\in L^p(G)$.
Then
\[
\gamma_j \lambda_p(s_j)(\xi)
= \gamma_j (\delta_{s_j}\ast\xi)
\xrightarrow{\|\cdot\|} \gamma (\delta_s\ast\xi)
= \gamma \lambda_p(s)(\xi),
\]
in $L^p(G)$.
This shows that $\gamma_j \lambda_p(s_j)\xrightarrow{SOT}\gamma \lambda_p(s)$.

Conversely, assume that $\gamma_j \lambda_p(s_j)\xrightarrow{SOT}\gamma \lambda_p(s)$.
To show that $s_j\to s$ in $G$, let $U$ be a neighborhood of the unit in $G$.
Choose a neighbourhood $V$ of the unit with $VV^{-1}\subseteq U$ and $\mu(V)<\infty$.
Then the characteristic function $\charFct_V$ belongs to $L^p(G)$.
We denote the left Haar measure with $\mu$.
If $s_j\notin sU$, then $s_jV\cap sV=\emptyset$, which we use to obtain
\begin{align*}
\| \gamma_j\lambda_p(s_j)(\charFct_V) - \gamma \lambda_p(s)(\charFct_V) \| &= \| \gamma_j \charFct_{s_jV} - \gamma \charFct_{sV} \|\\
&\geq \mu(s_jV \vartriangle sV)^{1/p}\\
&= (\mu(s_jV) + \mu(sV))^{1/p}.
\end{align*}
Since $\gamma_j \lambda_p(s_j)(\charFct_V) \xrightarrow{\|\cdot\|} \gamma \lambda_p(s)(\charFct_V)$, we deduce 
that the net $(s_j)_j$ eventually belongs to $sU$.
It follows that $s_j\to s$.
Similarly, one obtains that $\gamma_j\to\gamma$, as desired.
\end{proof}



For some applications of our results, particularly those in the next section, it will
be necessary to also compute the unitary group of some intermediate algebras between 
$\PF_p(G)$ and $\CV_p(G)$. In the next proposition, we assume that the unit ball of the
intermediate algebra is closed in the \weakStar-topology of $\Bdd(L^p(G))$, which is 
automatic whenever the algebra itself is \weakStar-closed, but the extra flexibility
will be necessary in the proof of \autoref{prp:ReflMain}.

\begin{prp}\label{cor:IntermediateSubalg}
Let $G$ be a locally compact group, and let $p\in(1,\infty)$ with $p\neq 2$.
Let $A$ be a Banach subalgebra of $\Bdd(L^p(G))$
such that $\PF_p(G)\subseteq A\subseteq\CV_p(G)$. Assume moreover that 
$A$ has a predual for which the inclusion $A\hookrightarrow \Bdd(L^p(G))$ is 
\weakStar-continuous. Then we have natural isomorphisms of topological groups
\[
\TT\times G \xrightarrow{\cong} \Unitary(A)_{w^*},\andSep
G \xrightarrow{\cong} \pi_0(\Unitary(A))_{w^*},
\]
given by $(\gamma,s)\mapsto\gamma \lambda_p(s)$, and by $s\mapsto [\lambda_p(s)]$, respectively.
\end{prp}
\begin{proof}
Denote by $(A_{\leq 1},w^*)$ the closed unit ball of $A$ endowed with the \weakStar-topology.
We claim that the canonical map 
$(A_{\leq 1},w^*)\to (\Bdd(L^p(G)),w^*)$ is a homeomorphism onto
its image. Continuity follows by assumption. Moreover, since $(A_{\leq 1},w^*)$ is compact, 
and $(\Bdd(L^p(G)),w^*)$ is Hausdorff, the map is automatically an homoemorphism onto
its image.

Using \autoref{prp:UCVp:TopUCVp}, it suffices to prove that $\mathcal{U}(\CV_p(G))$ is contained in $A$.
Let $(f_j)_j$ be a contractive approximate identity of $\PF_p(G)$ (obtained, for example, from a 
contractive approximate identity of $L^1(G)$). Fix $g\in G$. 
Then $\lambda_p(g)(f_j)$ is a contraction in $\PF_p(G)$,
and the \weakStar-limit of $\left(\lambda_p(g)(f_j)\right)_j$ is easily seen to be $\lambda_p(g)$. 
Since the unit ball of $A$ is \weakStar-closed, it follows that $\lambda_p(g)$ belongs to $A$, 
as desired. This finishes the proof.
\end{proof}


It has been conjectured by Herz (\cite{Her73SynthSubgps}) that $\PM_p(G)=\CV_p(G)$ for every locally compact group $G$.
While this conjecture remains open, the previous result implies that $\PM_p(G)$ and $\CV_p(G)$ have the same invertible isometries.
In particular, we deduce the following criterion for confirming Herz' conjecture.

\begin{cor}
Let $G$ be a locally compact group, and let $p\in(1,\infty)$ with $p\neq 2$. Then
$\PM_p(G)=\CV_p(G)$ if and only if the invertible isometries of $\CV_p(G)$ generate $\CV_p(G)$ as a \weakStar-closed
subalgebra of $\Bdd(L^p(G))$.
\end{cor}

If in \autoref{cor:IntermediateSubalg} one is not interested in recovering the topology of $G$ (or if its topology is trivial),
it suffices to consider subalgebras of $\CV(G)$ whose unit balls are not necessarily \weakStar-closed. 
We denote by $M_\lambda^p(G)$ the norm closure of $\lambda_p(M^1(G))$ in $\Bdd(L^p(G))$.

\begin{cor}
\label{prp:UCVp:UCVp_for_A}
Let $G$ be a locally compact group, and let $p\in(1,\infty)$ with $p\neq 2$.
Let $A$ be a norm-closed subalgebra of $\Bdd(L^p(G))$ such that $M^p_\lambda(G)\subseteq A\subseteq\CV_p(G)$.
Then we have natural group isomorphisms
\[
\TT\times G \xrightarrow{\cong} \Unitary(A),\andSep
G \xrightarrow{\cong} \pi_0(\Unitary(A)),
\]
given by $(\gamma,s)\mapsto\gamma \lambda_p(s)$, and by $s\mapsto [\lambda_p(s)]$, respectively.

In particular, a discrete group $G$ can be recovered from any closed subalgebra of $\Bdd(\ell^p(G))$ satisfying $\PF_p(G)\subseteq A\subseteq\CV_p(G)$ as
\[
G \cong \pi_0(\Unitary(A)).
\]
This applies in particular for the cases that $A$ is one of the following Banach algebras:
$\PF_p(G)$, $\PM_p(G)$, $\CV_p(G)$.
\end{cor}

\subsection{The unitary group of \texorpdfstring{$M_l(\PF_p(G))$}{Ml(Fpl(G))}}
\label{sec:UMFp}
In this subsection, we show that a locally compact group $G$ can be recovered from $\PF_p(G)$ by
considering the unitary group of the left multiplier algebra $M_l(\PF_p(G))$ with the strict topology.
We need to introduce some notation first.

\begin{ntn}
Let $A$ be a Banach algebra.
We let $M_l(A)$ denote the left multiplier algebra of $A$, which we identify with a closed subalgebra of $\Bdd(A)$ by
\[
M_l(A) = \big\{ S\in\Bdd(A) \colon  S(ab)=S(a)b \text{ for all } a,b\in A \big\}.
\]
We define $\pi_0(\Unitary(M_l(A)))$ as in \autoref{ntn:UCVp:pi0} with respect to the norm-topology on $\Unitary(M_l(A))$.

The strict topology on $M_l(A)$ is the restriction of the strong operator topology on $\Bdd(A)$ to $M_l(A)$, that is, for a net $(S_j)_j$ in $M_l(A)$ and $S\in M_l(A)$ we have $S_j\xrightarrow{\mathrm{str}} S$ if and only if $S_j(a)\xrightarrow{\|\cdot\|}S(a)$ for every $a\in A$.
We write $\Unitary(M_l(A))_{\mathrm{str}}$ for the group $\Unitary(M_l(A))$ equipped with the restriction of the strict topology.
We write $\pi_0(\Unitary(M_l(A)))_{\mathrm{str}}$ for the group $\pi_0(\Unitary(M_l(A)))$ equipped with the quotient topology induced by $\Unitary(M_l(A))_{\mathrm{str}}\to \pi_0(\Unitary(M_l(A)))$.
\end{ntn}

In the proof of the following result, we will use the unital, contractive homomorphism $\lambda_p\colon M^1(G)\to M^p_\lambda(G)$
given by left convolution. Note that $\lambda_p(s)=\lambda_p(\delta_s)$ for all $s\in G$.

\begin{thm}
\label{prp:UMFp:MainResult}
Let $G$ be a locally compact group, and let $p\in(1,\infty)$ with $p\neq 2$.
Then the maps
\[
\Lambda\colon\TT\times G\to\Unitary(M_l(\PF_p(G)))_{\mathrm{str}},\andSep
\Lambda'\colon G\to \pi_0(\Unitary(M_l(\PF_p(G)))_{\mathrm{str}},
\]
given by $\Lambda(\gamma,s)=\gamma \lambda_p(s)$, and by $\Lambda'(s)=[\lambda_p(s)]$, are isomorphisms of topological groups.
\end{thm}
\begin{proof}
Using \autoref{prp:inclusions}, we naturally identify $M_l(\PF_p(G))$ with a closed subalgebra of $\Bdd(L^p(G))$ satisfying $M^p_\lambda(G)\subseteq M_l(\PF_p(G))\subseteq\CV_p(G)$.
It follows from \autoref{prp:UCVp:UCVp_for_A} that $\Lambda$ and $\Lambda'$ are group isomorphisms.
Arguing as in the proof of \autoref{prp:UCVp:TopUCVp}, it is enough to show that $\Lambda$ is a homeomorphism.

To verify this, let $(\gamma_j,s_j)_j$ be a net in $\TT\times G$, and let $(\gamma,s)\in\TT\times G$.
Assume first that $(\gamma_j,s_j)\to(\gamma,s)$.
Given $f\in L^1(G)$, we have $\delta_{s_j}\ast f \xrightarrow{\|\cdot\|} \delta_s\ast f$ in $L^1(G)$.
Since $\lambda_p\colon L^1(G)\to \PF_p(G)$ is contractive, it follows that
\[
\gamma_j \lambda_p(s_j)(\lambda_p(f))
= \gamma_j \lambda_p( \delta_{s_j}\ast f )
\xrightarrow{\|\cdot\|} \gamma \lambda_p( \delta_{s}\ast f )
= \gamma \lambda_p(s)(\lambda_p(f))
\]
in $\PF_p(G)$.
Since the net $(\gamma_j \lambda_p(s_j))_j$ is bounded, and since $\lambda_p(L^1(G))$ is dense in $\PF_p(G)$, we deduce that $\gamma_j \lambda_p(s_j)(a) \xrightarrow{\|\cdot\|} \gamma \lambda_p(s)(a)$ in $\PF_p(G)$ for every $a\in \PF_p(G)$.
Thus, $\gamma_j \lambda_p(s_j) \xrightarrow{\mathrm{str}} \gamma \lambda_p(s)$ in $M_l(\PF_p(G))$.

Conversely, assume that $\gamma_j \lambda_p(s_j) \xrightarrow{\mathrm{str}} \gamma \lambda_p(s)$ in $M_l(\PF_p(G))$.
To show that $s_j\to s$, let $U$ be a neighbourhood of the unit in $G$.
Choose a neighbourhood $V$ of the unit with $V^2V^{-2}\subseteq U$ and $\mu(V)<\infty$.
Then $\charFct_V$ belongs to $L^1(G)$ and $L^p(G)$.
We have $\gamma_j \lambda_p(s_j)(\lambda_p(\charFct_V)) \xrightarrow{\|\cdot\|} \gamma \lambda_p(s)(\lambda_p(\charFct_V))$ in $\PF_p(G)$, and consequently
\[
\gamma_j \charFct_{s_jV^2}
= \gamma_j \lambda_p(s_j)(\lambda_p(\charFct_V))(\charFct_V) \xrightarrow{\|\cdot\|}
\gamma \lambda_p(s)(\lambda_p(\charFct_V))(\charFct_V)
= \gamma \charFct_{sV^2}
\]
in $L^p(G)$.
As in the proof of \autoref{prp:UCVp:TopUCVp} we deduce that $s_j$ eventually belongs to $sU$.
Thus, $s_j\to s$, and similarly $\gamma_j\to\gamma$.
\end{proof}

\section{The isomorphism problem for \texorpdfstring{$p$}{p}-convolution algebras}
\label{sec:MorCVp}

In this section, we prove our main result, \autoref{prp:MorCVp:IsoGpHq}, which completes
the proof of Theorem~\ref{thmintro:rigidity} from the introduction. Most of the work has been
done in the previous section, and what remains to do is to show that the ways in which the
topology of $G$ is recovered in \autoref{prp:UCVp:TopUCVp} and \autoref{prp:UMFp:MainResult}
only depend on the norm topology. While this is not very difficult for pseudofunctions,
it requires some work for convolvers and pseudomeasures.

Recall that a subset $S\subseteq A$ of a Banach algebra $A$ is said to be \emph{left essential}
if whenever $a\in A$ satisfies $ab=0$ for all $b\in S$, then $a=0$.

\begin{lma}
\label{prp:MorCVp:transferWkStarConv}
Let $A$ and $B$ be dual Banach algebras, 
and let $\varphi\colon A\to B$ be a (not necessarily isometric) isomorphism of Banach algebras.

Let $(a_j)_j$ be a bounded net in $A$ that converges \weakStar{} to $a$ in $A$.
Set
\[S=\{b\in A\colon \|a_jb-ab\|\to 0\},\]
and assume that $S$ is left essential.
Then the net $(\varphi(a_j))_j$ converges \weakStar{} to $\varphi(a)$.
\end{lma}
\begin{proof}
Since $\varphi$ is bounded, the net $(\varphi(a_j))_j$ is bounded in $B$.
To show that $(\varphi(a_j))_j$ converges \weakStar{} to $\varphi(a)$, it is enough to show that for every subnet of $(\varphi(a_j))_j$ that converges \weakStar{} to some $z\in B$ we have $z=\varphi(a)$.
Thus, we may assume that $(\varphi(a_j))_j$ converges \weakStar{} to some $z\in B$.
We will show that $z=\varphi(a)$.

Let $b\in S$.
Since multiplication by $\varphi(b)$ is a \weakStar{} continuous map on $B$, we have
\[
\varphi(a_j)\varphi(b) \xrightarrow{w^*} z\varphi(b).
\]
On the other hand, by definition of $S$, we have $a_jb \xrightarrow{\|\cdot\|} ab$ and therefore
\[
\varphi(a_j)\varphi(b) \xrightarrow{\|\cdot\|} \varphi(a)\varphi(b).
\]
Thus, $z\varphi(b)=\varphi(a)\varphi(b)$.
Applying $\varphi^{-1}$, we deduce that $\varphi^{-1}(z)b=ab$.
Since this holds for all $b\in S$, and since $S$ is left essential, we conclude that $\varphi^{-1}(z)=a$, and thus $z=\varphi(a)$, as desired.
\end{proof}


In the next lemma, for a unital Banach algebra $A$ we denote by $\Gl(A)$ the group of invertible elements in $A$.

\begin{lma}
\label{prp:MorCVp:MorFromCVpInducesGpHomo}
Let $G$ be a locally compact group, let $p\in(1,\infty)$, 
let $A\subseteq\Bdd(L^p(G))$ be a \weakStar-closed subalgebra containing $\PM_p(G)$
(for example $A=\PM_p(G)$, $A=\CV_p(G)$, or $A=\Bdd(L^p(G))$). Let $B$ be a dual Banach algebra, and let $\varphi\colon A\to B$ be an isomorphism of Banach algebras.
Then the map $\varphi\circ\lambda_p\colon G\to\Gl(B)_{w^*}$ is a continuous group homomorphism.
\end{lma}
\begin{proof}
Note that $\varphi\circ\lambda_p$ is clearly a group homomorphism.
To verify that it is continuous, let $(s_j)_j$ be a net that converges to $s$ in $G$.
Then $(\lambda_p(s_j))_j$ converges \weakStar{} to $\lambda_p(s)$ in $A$.

Set 
\[S=\{b\in A\colon \|\lambda_p(s_j)b-\lambda_p(s)b\|\to 0\}\subseteq A.\] 
We will verify that $S$ satisfies the assumptions of \autoref{prp:MorCVp:transferWkStarConv}.
Let $f\in L^1(G)$. 
Then
\[\lim_j\|\delta_{s_j}\ast f - \delta_s\ast f\|_1=0.\]
For each $j$, we have
\[
\|\lambda_p(s_j) \lambda_p(f) - \lambda_p(s) \lambda_p(f) \|
= \|\lambda_p(\delta_{s_j}\ast f - \delta_s\ast f)\|
\leq \|\delta_{s_j}\ast f - \delta_s\ast f\|_1,
\]
and therefore
$\lim_j \|\lambda_p(s_j) \lambda_p(f) - \lambda_p(s) \lambda_p(f) \| =0$.
Thus, $\lambda_p(f)$ belongs to $S$ for each $f\in L^1(G)$.

To verify that $S$ is left essential, let $x\in A$ satisfy $xb=0$ for every $b\in S$.
Let $(f_j)_j$ be a contractive approximate identity in $L^1(G)$.
By \autoref{prp:approxiIdCOnvId}, the net $(\lambda_p(f_j))_j$ converges \weakStar{} to $1$ in $\Bdd(L^p(G))$.
Since left multiplication by $x$ is \weakStar{} continuous, we deduce that $(x\lambda_p(f_j))_j$ converges \weakStar{} to $x$.
It follows that $x=0$, as desired.

Applying \autoref{prp:MorCVp:transferWkStarConv}, we deduce that $(\varphi(\lambda_p(s_j)))_j$ converges \weakStar{} to $\varphi(\lambda_p(s))$ in $B$, as desired.
\end{proof}

We are now ready to prove that $\CV_p(G)$ and $\PM_p(G)$ remember $G$.

\begin{thm}
\label{prp:IsoCVp:IsoWkClosed}
Let $G$ and $H$ be locally compact groups, and let $p\in(1,\infty)$ with $p\neq 2$.
Let $A\subseteq\Bdd(L^p(G))$ and $B\subseteq\Bdd(L^p(H))$ be \weakStar-closed subalgebras such that
\[
\PM_p(G) \subseteq A \subseteq \CV_p(G), \andSep
\PM_p(H) \subseteq B \subseteq \CV_p(H).
\]
Assume that $A$ and $B$ are isometrically isomorphic as Banach algebras.
Then $G$ and $H$ are isomorphic as topological groups.
\end{thm}
\begin{proof}
Let $\varphi\colon A\to B$ be an isometric isomorphism.
It induces a group isomorphism $\Unitary(A)\to\Unitary(B)$ that we also denote by $\varphi$.
Let $\Delta_A\colon\TT\times G\to\Unitary(A)_{w^*}$, $\Delta_A'\colon G\to\pi_0(\Unitary(A))_{w^*}$, $\Delta_B\colon\TT\times H\to\Unitary(B)_{w^*}$, 
and $\Delta_B'\colon H\to\pi_0(\Unitary(B))_{w^*}$ be the natural isomorphisms of topological groups from \autoref{cor:IntermediateSubalg}.

By \autoref{prp:MorCVp:MorFromCVpInducesGpHomo}, the map $G\to\Unitary(B)_{w^*}$ given by $s\mapsto\varphi(u_s)$ is continuous.
It follows that the map $\TT\times G\to\Unitary(B)_{w^*}$ given by $(\gamma,s)\mapsto\gamma\varphi(u_s)$ is continuous.
Precomposing with $\Delta_A^{-1}$, we deduce that $\varphi\colon\Unitary(A)_{w^*}\to\Unitary(B)_{w^*}$ is continuous.
The same argument applied to $\varphi^{-1}\colon B\to A$ shows that $\varphi\colon\Unitary(A)_{w^*}\to\Unitary(B)_{w^*}$ is an isomorphism of topological groups.

The map $\varphi\colon\Unitary(A)\to\Unitary(B)$ is continuous for the norm-topologies on $\Unitary(A)$ and $\Unitary(B)$.
Therefore, $\varphi$ induces a group isomorphism $\psi\colon\pi_0(\Unitary(A)) \to \pi_0(\Unitary(B))$ given by $\psi([u])=[\varphi(u)]$, for $u\in\Unitary(A)$.
In particular, the following diagram commutes:
\[
\xymatrix@R-5pt{
\TT \times G \ar[d]_{(\gamma,s)\mapsto s} \ar[r]^-{\Delta_A}
& \Unitary(A)_{w^*} \ar[d]^{u\mapsto[u]} \ar[r]^-{\varphi}
& \Unitary(B)_{w^*} \ar[d]^{u\mapsto[u]}
& \TT\times H \ar[d]^{(\gamma,s)\mapsto s} \ar[l]_-{\Delta_B} \\
G \ar[r]_-{\Delta_A'}
& { \pi_0(\Unitary(A)_{w^*} } \ar[r]_-{\psi}
& { \pi_0(\Unitary(B)_{w^*} }
& H \ar[l]^-{\Delta_B'}.
}
\]
The upper horizontal maps are isomorphisms of topological groups.
The downward maps are quotient maps.
It follows that $\psi$ is an isomorphism of topological groups.
Using that the lower horizontal maps $\Delta_A'$ and $\Delta_B'$ are isomorphisms of topological groups, we deduce that $G$ and $H$ are isomorphic as topological groups.
\end{proof}


The proof that $\PF_p(G)$ remembers $G$ is somewhat less complicated.

\begin{lma}
\label{prp:IsoCVp:IsoFp}
Let $G$ and $H$ be locally compact groups, and let $p\in(1,\infty)$ with $p\neq 2$.
Assume that $\PF_p(G)$ and $\PF_p(H)$ are isometrically isomorphic as Banach algebras.
Then $G$ and $H$ are isomorphic as topological groups.
\end{lma}
\begin{proof}
The proof is similar to that of \autoref{prp:IsoCVp:IsoWkClosed}.
Let $\varphi\colon \PF_p(G)\to \PF_p(H)$ be an isometric isomorphism.
Then $\varphi$ induces an isometric isomorphism
\[\Phi\colon M_l(\PF_p(G))\to M_l(\PF_p(H)).\]
Set $A=M_l(\PF_p(G))$ and $B=M_l(\PF_p(H))$.
Note that $\Phi$ and $\Phi^{-1}$ are both norm-continuous and strictly continuous.

The map $\Phi$ induces a group isomorphim $\alpha\colon \Unitary(A) \to \Unitary(B)$.
Since $\alpha$ is a homeomorphism for the norm-topologies, it induces a group isomorphism $\psi\colon\pi_0(\Unitary(A)) \to \pi_0(\Unitary(B))$.

Since $\Phi$ is a homeomorphism for the strict topologies, we deduce that
\[\alpha\colon \Unitary(A)_{\mathrm{str}} \to \Unitary(B)_{\mathrm{str}}\]
is an isomorphism of topological groups.
It follows that
\[\psi\colon \pi_0(\Unitary(A))_{\mathrm{str}} \to \pi_0(\Unitary(B))_{\mathrm{str}}\]
is an isomorphism of topological groups.
Let
\[\Lambda_A'\colon G\to\pi_0(\Unitary(A))_{\mathrm{str}} \ \mbox{ and } \ \Lambda_B'\colon H\to\pi_0(\Unitary(B))_{\mathrm{str}}\]
be the natural isomorphisms of topological groups from \autoref{prp:UMFp:MainResult}.
Then the following maps are isomorphisms
\[
\xymatrix{G \ar[r]^-{\Lambda_B'} & \pi_0(\Unitary(A))_{\mathrm{str}}
\ar[r]^-{\psi}& \pi_0(\Unitary(B))_{\mathrm{str}}
&
H\ar[l]_-{\Lambda_A'},}
\]
which shows that $G$ and $H$ are isomorphic as topological groups.
\end{proof}


Putting these results together, we arrive at one of the main statements of our work.

\begin{cor}
\label{prp:MorCVp:CharMorCVp}
Let $G$ and $H$ be locally compact groups, and let $p\in(1,\infty)$ with $p\neq 2$.
Then the following are equivalent:
\begin{enumerate}
\item
There is an isomorphism of topological groups $G\cong H$;
\item
There is an isometric isomorphism $\PF_p(G)\cong \PF_p(H)$;
\item
There is an isometric isomorphism $\PM_p(G)\cong\PM_p(H)$;
\item
There is an isometric isomorphism $\CV_p(G)\cong\CV_p(H)$.
\end{enumerate}
\end{cor}
\begin{proof}
It is clear that~(1) implies~(2), (3) and~(4).
It follows from \autoref{prp:IsoCVp:IsoWkClosed} that~(3) implies~(1), and that~(4) implies~(1).
Further, we obtain from \autoref{prp:IsoCVp:IsoFp} that~(2) implies~(1).
\end{proof}

In \cite[Corollary~4.14]{GarThi18:ReprConvLq}, we showed that for every locally compact group $G$, and $p,q\in(1,\infty)$, the following are equivalent:
\begin{enumerate}
\item
$\left| \tfrac{1}{p} - \tfrac{1}{2} \right| = \left| \tfrac{1}{q} - \tfrac{1}{2} \right|$.
\item
There is an isometric (anti-)isomorphism $\PF_p(G) \cong \PF_q(G)$;
\item
There is an isometric (anti-)isomorphism $\PM_p(G) \cong \PM_q(G)$;
\item
There is an isometric (anti-)isomorphism $\CV_p(G) \cong \CV_q(G)$;
\end{enumerate}	
The first condition means that $p$ and $q$ are either equal or dual H\"older exponents.

Combined with \autoref{prp:MorCVp:CharMorCVp}, we obtain the following strong rigidity result.
It shows that the Banach algebras of $p$-pseudofunctions, of $p$-pseudomeasures, and of $p$-convolvers remember both the locally group and the H\"older exponent.

\begin{cor}
\label{prp:MorCVp:IsoGpHq}
Let $G$ and $H$ be locally compact groups, and let $p,q\in(1,2)$. 
Then the following are equivalent:
\begin{enumerate}
\item
We have $p=q$, and there is an isomorphism of topological groups $G\cong H$;
\item
There is an isometric isomorphism $\PF_p(G)\cong \PF_q(H)$;
\item
There is an isometric isomorphism $\PM_p(G)\cong\PM_q(H)$;
\item
There is an isometric isomorphism $\CV_p(G)\cong\CV_q(H)$.
\end{enumerate}
\end{cor}

\begin{rmk}
Let $G$ be a locally compact group, let $p\in(1,\infty)$, and let $p'$ be the dual H\"older exponent.
By \cite[Proposition~4.11]{GarThi18:ReprConvLq}, there are natural isometric \emph{anti-}isomorphisms
\[
\PF_p(G) \cong \PF_{p'}(G), \andSep
\PM_p(G)\cong\PM_{p'}(G), \andSep
\CV_p(G)\cong\CV_{p'}(G).
\]
Using this, we obtain a version of \autoref{prp:MorCVp:IsoGpHq} that allows for $p,q\in(1,\infty)\setminus\{2\}$, where~(1) is replaced by the condition that $\big| \tfrac{1}{p} - \tfrac{1}{2} \big| = \big| \tfrac{1}{q} - \tfrac{1}{2} \big|$ and $G\cong H$; and where in~(2), (3) and (4) one has to allow for anti-isomorphisms.
\end{rmk}

\subsection{Other isomorphism results for convolution operators}
Wendel's appraised result motivated a number of authors to search for $L^p$ analogs of it. A series of
intermediate results culminated in Parrott's theorem (\cite[Theorem~2]{Par68IsoMultiplier}) asserting that for locally
compact groups $G$ and $H$, and for $p\in (1,\infty)\setminus\{2\}$, if there is a linear isometry $T\colon L^p(G)\to L^p(H)$
such that $T(\xi\ast\eta)=T(\xi)\ast T(\eta)$ whenever $\xi,\eta,\xi\ast\eta\in L^p(G)$, and similarly for $T^{-1}$,
then $G$ and $H$ are isomorphic. In \autoref{prp:Parrott}, we show that our \autoref{prp:IsoCVp:IsoWkClosed} is formally
stronger, by deducing Parrott's result as a consequence of it.
Another reason to reprove Parrott's result is the fact that it relies on Lamperti's imprecise description of
the invertible isometries of a $\sigma$-finite $L^p$-space (see our comments before \autoref{prp:UCVp:identifyUCVp}).

Let $G$ be a locally compact group, and let $p\in(1,\infty)$.
We say that $f\in L^p(G)$ is a left multiplier if $f\ast\xi$ belongs to $L^p(G)$ for each $\xi\in L^p(G)$, and if
the left convolution map $\lambda_p(f)\colon L^p(G)\to L^p(G)$ is bounded.

\begin{lma}
\label{prp:partProdWkClosure}
Let $G$ be a locally compact group, and let $p\in(1,\infty)$. 
Let $A$ denote the \weakStar-closure in $\Bdd(L^p(G))$ of the set $\{\lambda_p(f) \colon f\in L^p(G) \text{ is left multiplier} \}$.
Then $A$ is a \weakStar-closed subalgebra of $\Bdd(L^p(G))$ satisfying $\PM_p(G)\subseteq A\subseteq\CV_p(G)$.
\end{lma}
\begin{proof}
Set $A_0= \{f \in L^p(G)\colon f \text{ is left multiplier} \}$.
Given $f,g\in A_0$, we have $f\ast g\in A_0$, which implies that $A$ is a subalgebra of $\Bdd(L^p(G))$.

Recall that $\rho\colon G\to\Bdd(L^p(G))$ denotes the right regular representation.
Given $\xi\in L^p(G)$ and $s\in G$, we have
\[
(\rho_s\lambda_p(f))(\xi)
= \rho_s(f\ast\xi)
= f\ast\xi\ast\delta_{s^{-1}}
= \lambda_p(f)(\xi\ast\delta_{s^{-1}})
= (\lambda_p(f)\rho_s)(\xi).
\]
This shows that $\lambda_p(f)$ belongs to $\CV_p(G)$, and thus $A\subseteq\CV_p(G)$.

Lastly, since $C_c(G)$ is contained in $A_0$, it follows that $A$ contains the \weakStar-closure of $\lambda_p(C_c(G))$.
Said closure is easily seen to be $\PM_p(G)$, using that $C_c(G)$ is norm-dense in $L^1(G)$, and the proof is complete.
\end{proof}

We now show how to reprove \cite[Theorem~2]{Par68IsoMultiplier}.

\begin{thm}
\label{prp:Parrott}
Let $G$ and $H$ be locally compact groups, and let $p\in(1,\infty)$ with $p\neq 2$.
Let $T\colon L^p(G)\to L^p(H)$ be a surjective, linear isometry such that $T(f\ast g)=T(f)\ast T(g)$
whenever $f,g,f\ast g\in L^p(G)$, and such that $T^{-1}(f\ast g)=T^{-1}(f)\ast T^{-1}(g)$ whenever $f,g,f\ast g\in L^p(H)$.
Then $G$ and $H$ are isomorphic as topological groups.
\end{thm}
\begin{proof}
As in the proof of \autoref{prp:partProdWkClosure}, we set
\begin{align*}
A_0 = \{f \in L^p(G) \colon f \text{ is left multiplier} \},\andSep
A = \overline{ \lambda_p(A_0) }^{w^*} \subseteq \Bdd(L^p(G)),
\end{align*}
and define $B_0\subseteq L^p(H)$ and $B\subseteq \Bdd(L^p(H))$ analogously, using $H$ in place of $G$.
Then $T$ induces an isometric isomorphism of Banach algebras
\[
T_*\colon\Bdd(L^p(G))\to\Bdd(L^p(H)),
\]
given by $T_*(a)=T\circ a\circ T^{-1}$ for $a\in\Bdd(L^p(G))$.

Next, we show that $T_*$ restricts to an isometric isomorphism of Banach algebras $A\to B$.
Let $f\in A_0$.
To show that $T_*(\lambda_p(f))=\lambda_p(T(f))$, let $\xi\in L^p(H)$.
Using the assumption on $T$ at the third step, we get
\[
T_*(\lambda_p(f))(\xi)
= (T\circ\lambda_p(f)\circ T^{-1})(\xi)
= T( f\ast T^{-1}(\xi) )
= T(f) \ast \xi
= \lambda_p(T(f))(\xi),
\]
as desired.
Thus, $T_\ast$ maps $\lambda_p(A_0)$ to $\lambda_p(B_0)$.
Since $T_\ast$ is \weakStar-continuous, we get that $T_\ast$ maps $A$ to $B$.
The same argument applied for $T^{-1}$ shows that $T_\ast$ maps $A$ isometrically onto $B$.
Finally, it follows from \autoref{prp:IsoCVp:IsoWkClosed} that $G$ and $H$ are isomorphic as topological groups.
\end{proof}

\section{An application: the reflexivity conjecture for \texorpdfstring{$p$}{p}-convolution algebras}

In this final section, we apply the methods developed in \autoref{sec:UnitGpConvAlg} to confirm a
conjecture of Gal\'e-Ransford-White \cite{GalRanWhi_weakly_1992} for $p$-convolution algebras.
Their conjecture asserts that every reflexive, amenable Banach algebra is automatically finite
dimensional, and it has been confirmed in a number of situations; see, for example, the introduction
of \cite{Run_Banach_2001}. We begin with some preparatory results, which are interesting on their own right. 

\begin{prp}\label{prp:UnitGpTopGp}
Let $A$ be a unital, dual Banach algebra.
Then $\mathcal{U}(A)$ is a topological group for the \weakStar-topology.
\end{prp}
\begin{proof}
Use \cite[Corollary~3.8]{Daw07DualBAlgReprInj} to find a reflexive Banach space $E$ and an isometric,
\weakStar-continuous homomorphism $\varphi\colon A\to\Bdd(E)$ that identifies $A$ with a \weakStar-closed subalgebra of $\Bdd(E)$.

\textbf{Claim:} \emph{We may assume that $\varphi$ is unital}. Indeed, $\varphi(1)$ is a contractive idempotent in $\Bdd(E)$,
and the subspace $E_0=\varphi(1)(E)$ is also reflexive. Moreover, the induced representation $\varphi_0\colon A\to\Bdd(E_0)$
is unital.

Then $\varphi$ maps $\mathcal{U}(A)$ onto a subgroup of $\Isom(E)$.
Moreover, this identifies the \weakStar-topology on $\Unitary(A)$ with the \weakStar-topology on $\varphi(\Unitary(A))$ coming from $\Bdd(E)$.
As noted in \autoref{pgr:topologiesBddE}, the \weakStar-topology and the SOT-topology agree on $\Isom(E)$.
Thus, $\varphi$ induces an isomoprhism $\Unitary(A)_{w^*}\cong\varphi(\Unitary(A))_{SOT}$ of topologcal groups.
Since $\Isom(E)_{SOT}$ is a topological group, so is its subgroup $\varphi(\Unitary(A))_{SOT}$ with the relative topology, and hence also $\Unitary(A)_{w^*}$.
\end{proof}

\begin{thm}
\label{prp:normWkTopReflDual}
Let $A$ be a unital, dual Banach algebra that is reflexive as a Banach space.
Then the norm topology and the \weakStar-topology agree on $\Unitary(A)$.
\end{thm}
\begin{proof}
We consider the representation $\varphi\colon A\to\Bdd(A)$ given by $\varphi(a)(b)=ab$, for $a,b\in A$.
Let $(u_j)_j$ be a net in $\Unitary(A)$, and let $u\in\Unitary(A)$.

\textbf{Claim~1:}
\emph{We have $u_j\xrightarrow{w^*}u$ if and only if $\varphi(u_j)\xrightarrow{WOT}\varphi(u)$.}
The proof follows from the following sequence of equivalences, using at the first step that multiplication in $A$ is
separately \weakStar-continuous and that $A$ is unital, and using at the second step that weak and \weakStar-convergence on $A$ agree: 
\begin{align*}
u_j\xrightarrow{w^*}u
&\quad\Leftrightarrow\quad u_ja\xrightarrow{w^*}ua \text{ for all } a\in A \\
&\quad\Leftrightarrow\quad \varphi(u_j)(a) \xrightarrow{w} \varphi(u)(a) \text{ for all } a\in A \\	
&\quad\Leftrightarrow\quad \varphi(u_j)\xrightarrow{WOT}\varphi(u).
\end{align*}

\textbf{Claim~2:}
\emph{We have $u_j\xrightarrow{\|\cdot\|}u$ if and only if $\varphi(u_j)\xrightarrow{SOT}\varphi(u)$.}
The claim is proved analogously to Claim~1, and we omit it.

Since $A$ is reflexive, as noted in \autoref{pgr:topologiesBddE}, the WOT- and SOT-topology agree on $\Isom(A)$.
Using Claims~1 and~2, the result follows.
\end{proof}


We need to recall the definition of amenability for Banach algebras. (What we give here is really 
an equivalent characterization, due to Johnson~\cite{Joh72CohomologyBAlg}.)

\begin{dfn}
Let $A$ be a Banach algebra, and denote by $\pi\colon A\to A\widehat{\otimes}A$ the 
diagonal map. We say that $A$ is \emph{amenable} if there exists
$m\in (A\widehat{\otimes}A)^{**}$ satisfying $am=ma$ and $\pi^{\ast\ast}(m)a=a$
for all $a\in A$.
\end{dfn}

The following is the main result of this section, and confirms the reflexivity conjecture
for all Banach subalgebras of $\CV_p(G)$ containing $\PF_p(G)$. This includes, in particular,
the convolution algebras $\PF_p(G)$, $\PM_p(G)$ and $\CV_p(G)$. Its proof uses results
and ideas of a number of authors, including Herz, Phillips and Runde. The results
in Section~4 are used to show that $G$ is discrete, while we also use results from
\cite{GarThi15GpAlgLp} to prove that $G$ is compact. 

In its proof, we will need the Fig\`{a}-Talamanca-Herz algebra $A_p(G)$, which is defined
as follows. With 
\[\omega\colon L^p(G)\tensMax L^{p'}(G)\to C_0(G)\] 
denoting the linear map satisfying
$\omega(\xi\otimes\eta)(s) = \langle \lambda_p(s)\xi,\eta \rangle$,
for $\xi\in L^p(G)$, $\eta\in L^{p'}(G)$, and $s\in G$, the algebra $A_p(G)$ is the image of 
$\omega$, endowed with the quotient norm and the product inherited from $C_0(G)$. It is
a standard fact in the area that $A_p(G)$ is a canonical isometric predual of $\PM_p(G)$, 
turning $\PM_p(G)$ into a dual Banach algebra. 
Moreover, under the identifications of $\PM_p(G)$ with the dual of $A_p(G)$, and of $\Bdd(L^p(G))$ with the 
dual of $L^p(G)\widehat{\otimes} L^{p'}(G)$, the inclusion $\PM_p(G)\to\Bdd(L^p(G))$ is the transpose of $\omega$.


\begin{thm}
\label{prp:ReflMain}
Let $G$ be a locally compact group, and let $p\in(1,\infty)$.
Let $A$ be a Banach algebra satisfying $\PF_p(G)\subseteq A\subseteq \CV_p(G)$.
Then the following are equivalent:
\begin{enumerate}
 \item $A$ is reflexive and amenable;
 \item $A$ is finite-dimensional;
 \item $G$ is finite. 
\end{enumerate}
\end{thm}
\begin{proof}
That (2) and (3) are equivalent is clear. Moreover, when $G$ is finite, then
$\PF_p(G)=\CV_p(G)$ is finite-dimensional and amenable, showing that (3) implies
(1). It remains to show that (1) implies (3), and we divide its proof into a 
number of claims. Since the result is well-known for $p=2$, we assume that $p\neq 2$. 

We first note that the inclusion $A\to \Bdd(L^p(G))$ is
\weakStar-continuous, since this map is the adjoint of the following composition
\[L^p(G)\widehat{\otimes}L^{p'}(G)\to (L^p(G)\widehat{\otimes}L^{p'}(G))^{**}\cong \Bdd(L^p(G))^* \to A^*,\]
once $A^{**}$ is canonically identified with $A$. 

\textbf{Claim 1:} \emph{$G$ is discrete.} It is well-known that $A$ has a unit,
being amenable and reflexive. By \autoref{prp:normWkTopReflDual}, there is a 
canonical group isomorphism between $\mathcal{U}(A)_{w^*}$ and 
$\mathcal{U}(A)_{\|\cdot\|}$. Hence there is a group isomorphism
$\pi_0(\mathcal{U}(A))_{w^*}\cong \pi_0(\mathcal{U}(A))_{\|\cdot\|}$. 
By \autoref{cor:IntermediateSubalg}, there is a
group isomorphism $\pi_0(\mathcal{U}(A))_{w^*}\cong G$, while 
Equation~\ref{prp:UCVp:identifyUCVp:eqNorm} in \autoref{prp:UCVp:identifyUCVp}
implies that $\pi_0(\mathcal{U}(A))_{\|\cdot\|}$ is isomorphic to $G$
with the discrete topology. The claim follows.

The proof of the following claim is inspired by work of Chris Phillips~\cite{Phi14pre:LpMultDom},
and does not require reflexivity of $A$. We denote by $\rho_p\colon \ell^1(G)\to \Bdd(\ell^p(G))$ the
right regular representation.

\textbf{Claim 2:} \emph{$G$ is amenable.} Observe that the commutant $A^c$ of $A$ in $\Bdd(\ell^p(G))$ agrees
with the double commutant of $\rho_p(\ell^1(G))$. In particular, $A^c$ is anti-isomorphic to $\CV_p(G)$,
and therefore there exists a bounded, linear map $\tau\colon A^c\to \mathbb{C}$ satisfying $\tau(\rho_p(g)b\rho_p(g^{-1}))=\tau(b)$
for all $g\in G$ and all $b\in A^c$. Use Proposition~4.4 in~\cite{Phi14pre:LpMultDom} and the fact that $G$ is discrete
to find a unital bounded linear map $E\colon \Bdd(\ell^p(G))\to A^c$ satisfying $E(bxc)=bE(x)c$ for all $b,c\in A^c$ and all 
$x\in \Bdd(\ell^p(G))$. With $m\colon \ell^{\infty}(G)\to \Bdd(\ell^p(G))$ denoting the canonical embedding as
multiplication operators, it follows that $\tau\circ E \circ m\colon \ell^\infty(G)\to \mathbb{C}$ is a non-zero, translation-invariant
bounded linear map. In particular, $\tau\circ E \circ m/\|\tau\circ E \circ m\|$ is an invariant mean, and $G$ is amenable.

\textbf{Claim 3:} \emph{$A_p(G)$ is relexive}. Note first that $\PF_p(G)$ is reflexive, 
being a Banach subspace of $A$. Recall that $F^p_{\mathrm{QS}}(G)$ denotes the universal completion
of $L^1(G)$ with respect to isometric representations of $G$ on $QSL^p$-spaces; 
see~\cite[Definition~2.1 and Remark~2.9]{GarThi15GpAlgLp}. By \cite[Corollary 5.2]{Run_representations_2005},
and since $G$ is amenable, there is a natural isometric homomorphism $A_p(G) \hookrightarrow (F^p_{\mathrm{QS}}(G))^*$.
Moreover, and again since $G$ is amenable, there is a natural isometric isomorphism $F^p_{\mathrm{QS}}(G)\cong \PF_p(G)$, 
by \cite[Theorem~3.7]{GarThi15GpAlgLp}. The following composition
\[A_p(G)\hookrightarrow (F^p_{\mathrm{QS}}(G))^*\cong (\PF_p(G))^*\] 
shows that $A_p(G)$ is a Banach subspace of the reflexive space $(\PF_p(G))^*$, and it is therefore reflexive. 

The next claim completes the proof of the theorem.

\textbf{Claim 4:} \emph{$G$ is compact.} 
By the Leptin-Herz theorem (see \cite[Theorem~10.4]{Pie_amenable_1984}), $A_p(G)$ contains a bounded approximate identity.
It follows that the bidual $A_p(G)^{**}$ contains a mixed unit (for the left and right Arens products).
Using that $A_p(G)$ is reflexive, we deduce that 
$A_p(G)^{**}=A_p(G)\subseteq C_0(G)$ contains the constant function with value $1$. Thus $G$ is compact.
\end{proof}


\end{document}